\documentclass[pdflatex,sn-mathphys-ay]{sn-jnl}

\usepackage{graphicx}%
\usepackage{multirow}%
\usepackage{amssymb}
\usepackage{amsmath}
\usepackage{amsfonts}
\usepackage{mathrsfs}
\usepackage{mathtools}
\usepackage[title]{appendix}%
\usepackage{xcolor}%
\usepackage{textcomp}%
\usepackage{manyfoot}%
\usepackage{booktabs}%
\usepackage{algorithm}%
\usepackage{algorithmicx}%
\usepackage{algpseudocode}%
\usepackage{listings}%
\usepackage{enumitem}

\usepackage{tikz}
\usetikzlibrary{arrows.meta,positioning,fit,calc,decorations.pathreplacing,shapes.multipart}
\usepackage{cleveref}

\theoremstyle{thmstyleone}
\newtheorem{theorem}{Theorem}[section]
\newtheorem{proposition}[theorem]{Proposition}
\newtheorem{lemma}[theorem]{Lemma}

\theoremstyle{definition}
\newtheorem{definition}[theorem]{Definition}

\theoremstyle{remark}
\newtheorem{remark}[theorem]{Remark}

\begin{document}

\title{Full-Covariance Chemical Langevin Predator–Prey Diffusion with Absorbing Boundaries}

\author[1]{\fnm{Jiguang} \sur{Yu}}\email{jyu678@bu.edu}
\equalcont{These authors contributed equally to this work as co-first authors.}

\author*[2]{\fnm{Louis Shuo} \sur{Wang}}\email{swang116@vols.utk.edu}
\equalcont{These authors contributed equally to this work as co-first authors.}

\author[3]{\fnm{Yuansheng} \sur{Gao}}\email{y.gao@zju.edu.cn}

\author[4]{\fnm{Ye} \sur{Liang}}\email{ye-liang@uiowa.edu}

\affil[1]{\orgdiv{College of Engineering},
  \orgname{Boston University},
  \orgaddress{\city{Boston}, \postcode{02215}, \state{MA}, \country{USA}}}

\affil[2]{\orgdiv{Department of Mathematics},
  \orgname{University of Tennessee},
  \orgaddress{\city{Knoxville}, \postcode{37996}, \state{TN}, \country{USA}}}

\affil[3]{\orgdiv{College of Computer Science and Technology},
  \orgname{Zhejiang University},
  \orgaddress{\city{Hangzhou}, \postcode{310000}, \state{Zhejiang},\country{China}}}

\affil[4]{\orgdiv{Department of Industrial and Systems Engineering},
  \orgname{The University of Iowa},
  \orgaddress{\city{Iowa City}, \postcode{52242}, \state{IA},\country{USA}}}

\abstract{
Many stochastic predator--prey models for Rosenzweig--MacArthur dynamics add ad hoc independent (diagonal)
noise, which cannot encode the event-level coupling created by predation and biomass conversion. We develop a fully
mechanistic route from elementary reactions to an absorbed diffusion approximation and its extinction
structure. Starting from a continuous-time Markov chain on $\mathbb N_0^2$ with four reaction channels, including prey birth,
prey competition death, predator death, and a coupled predation--conversion event, we impose absorbing coordinate axes,
capturing the irreversibility of demographic extinction. Under Kurtz density-dependent scaling, the
law-of-large-numbers limit recovers the classical Rosenzweig--MacArthur ordinary differential equation, while central-limit scaling
yields a chemical-Langevin diffusion with explicit drift and a full state-dependent covariance matrix. A key
structural feature is the strictly negative instantaneous cross-covariance
$\Sigma_{12}(N,P)=-mNP/(1+N)$, arising uniquely from the predation--conversion increment $(-1,1)$ and absent from
diagonal-noise surrogates. We provide two equivalent Brownian representations: an event-based factorization with a
four-dimensional driver preserving reaction channels, and a two-dimensional Cholesky factorization convenient for
simulation. Defining the absorbed It\^o SDE by freezing trajectories at the first boundary hit, we prove strong
existence, pathwise uniqueness, non-explosion, and uniform moment bounds up to absorption. Extinction occurs with
positive probability from every interior state, and predator extinction is almost sure in the subcritical regime
$m\le c$. Reproducible Euler--Maruyama diagnostics validate factorization consistency and quantify the effect of the
cross-covariance in representative parameter regimes.
}

\maketitle

\section{Introduction}
\label{sec:intro}

Extinction is an intrinsically irreversible event in finite populations: once one species' count
hits zero, there is no endogenous mechanism that recreates it.
At the individual (demographic) level, this irreversibility is naturally encoded by a continuous-time Markov chain (CTMC) on $\mathbb{N}_0^2$. In this model, the coordinate axes are absorbing sets. There are no births in the absence of individuals, and no predation or conversion events unless both species are present.

\begin{figure}[htbp]
\centering
\begin{tikzpicture}[
    font=\footnotesize, 
    >=Latex,
    node distance=1.0cm and 0.8cm,
    box/.style={
        draw, rounded corners=2pt, align=left, 
        inner sep=5pt, 
        text width=4.0cm, 
        minimum height=1.8cm, 
        fill=white
    },
    proc/.style={box, fill=gray!10},
    theo/.style={box, fill=gray!5, dashed}, 
    key/.style={
        draw, rounded corners=2pt, align=left, 
        inner sep=5pt, text width=4.0cm, 
        line width=1.2pt, fill=white, 
        minimum height=1.8cm
    },
    arrow/.style={->, line width=0.8pt},
    dasharrow/.style={->, dashed, line width=0.8pt},
    note/.style={draw=none, align=center, inner sep=1pt, font=\scriptsize\bfseries}
]

\node[proc] (events) {
    \textbf{Elementary events}\\[2pt]
\(\mathcal E=\{B,C,D,E\}\)\\
\(\Delta_B=(1,0)\), \(\Delta_C=(-1,0)\)\\
\(\Delta_D=(0,-1)\), \(\Delta_E=(-1,1)\)\\
\(\lambda_E(N,P)=\dfrac{mNP}{1+N}\)};

\node[proc, right=of events] (ctmc) {
    \textbf{Mechanistic CTMC}\\[2pt]
    State space: $\mathbb N_0^2$\\
    Random time-change\\
    Axes absorbing $\Rightarrow$ extinction irreversible
};

\node[proc, right=of ctmc] (scaling) {
    \textbf{Kurtz scaling}\\[2pt]
    $Z^\Omega=\Omega^{-1}X^\Omega$\\
    Noise scale $\rho=\Omega^{-1/2}$\\
    Langevin limit: LLN \& CLT
};

\node[proc, below=of ctmc] (ode) {
    \textbf{LLN: mean-field ODE}\\[2pt]
    Rosenzweig--MacArthur drift\\
    $\dot z=\mu(z)$\\
    Coexistence / Hopf / Limit cycles
};

\node[key, below=of scaling] (sigma) {
    \textbf{CLT: Full Covariance}\\[2pt]
    $\Sigma(z)=\sum \lambda_e \Delta_e \Delta_e^\top$\\
    \textbf{Mechanistic fingerprint:}\\
    $\Sigma_{12} < 0$ (Coupling)
};

\node[proc, below=of sigma] (factors) {
    \textbf{Factorizations}\\[2pt]
    Event: $\Sigma=L_{\rm ev}L_{\rm ev}^\top$\\
    Cholesky: $\Sigma=L_{\rm chol}L_{\rm chol}^\top$
};

\node[proc, left=of factors] (sde) {
    \textbf{Interior SDE on } $U$\\[2pt]
    $dZ=\mu\,dt+\rho L\,dW$\\
    Valid while $Z(t)\in (0,\infty)^2$
};

\node[proc, left=of sde] (absorb) {
    \textbf{Boundary Absorption}\\[2pt]
\(\tau=\inf\{t:\,N(t)=0 \text{ or } P(t)=0\}\)\\
Freeze: \(\widehat Z(t)=Z(t\wedge\tau)\)
};

\node[theo, below=0.5cm of absorb, text width=4cm, minimum height=1cm] (ext) {
    \textbf{Theory:}\\
    $\mathbb P(\tau<\infty)>0$ always\\
    Predator dies if $m \le c$
};

\node[proc, below=0.5cm of factors, text width=4cm, minimum height=1cm] (num) {
    \textbf{Numerics:}\\
    Absorbed Euler-Maruyama\\
    Full vs Diag covariance
};

\node[note, above left=0mm of sigma, anchor=south east, xshift=20mm, yshift=0mm] (tag) {(Key Novelty)};

\draw[arrow] (events) -- (ctmc);
\draw[arrow] (ctmc) -- (scaling);

\draw[arrow] (scaling) -- (sigma);

\draw[arrow] (scaling.south) -- ++(0, -0.35) -| (ode.north);

\draw[arrow] (sigma) -- (factors);

\draw[arrow] (factors) -- (sde);
\draw[arrow] (sde) -- (absorb);

\draw[dasharrow] (absorb) -- (ext);
\draw[dasharrow] (factors) -- (num);
\draw[dasharrow] (sde) -- ++(-1.2, -1.2) node[anchor=north east, font=\scriptsize] {Well-posedness};

\node[draw, dashed, inner sep=10pt, rounded corners=5pt, line width=1pt, color=gray!80,
      fit=(events) (scaling) (num) (ext),
      label={[font=\bfseries, yshift=-3ex]north:Mechanistic Pipeline}] {};

\end{tikzpicture}
\caption{\textbf{Mechanistic roadmap.} The pipeline proceeds from event-level definition to macroscopic scaling, yielding a full-covariance diffusion with structural signature $\Sigma_{12}<0$.}
\label{fig:roadmap}
\end{figure}
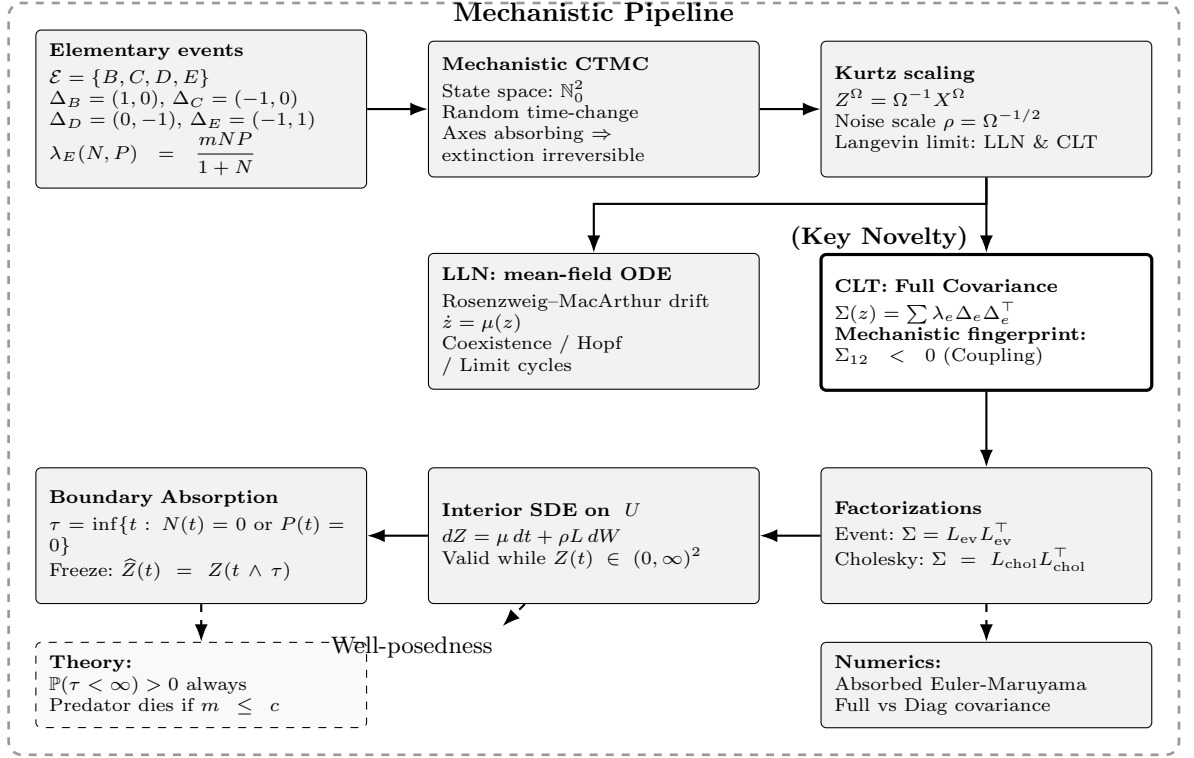

Such absorption and extinction mechanisms are standard in stochastic population models; see, for example, the general discussions in \citep{allen_stochastic_2015,ovaskainen_stochastic_2010,nasell_extinction_2001,lande_stochastic_2003}.
Such Markov chain models are frequently investigated through sample-path simulations to obtain Monte Carlo estimates of quantities of interest \citep{gibson_efficient_2000,gillespie_general_1976,korner-nievergelt_markov_2015,hobolth_simulation_2009}.
In the diffusion approximation developed here, we adopt the same biological convention by imposing absorption at the coordinate axes. Once the diffusion reaches $\{N=0\}\cup \{P=0\}$, it is frozen at that state. Consequently, the boundary represents demographic extinction events, rather than reflecting or re-entering behavior.

Deterministically, the Rosenzweig--MacArthur predator--prey system with a Holling type~II functional response
provides a canonical mean-field backbone for coexistence and predator persistence thresholds
\citep{rosenzweig_graphical_1963,holling_characteristics_1959,murray_mathematical_2002}.
In the coexistence regime, the model is also a standard setting for enrichment-driven destabilization
and oscillatory dynamics \citep{rosenzweig_paradox_1971,sugie_uniqueness_2012}.
However, when demographic stochasticity is incorporated, the form of the noise becomes model-defining:
a mechanistic description must reflect that a single predation--conversion event simultaneously decreases prey
and increases predators. Consequently, any diffusion approximation derived from such events carries an
instantaneous negative cross-covariance in its demographic noise, which is not captured by ad hoc
diagonal/independent-noise surrogates.
This observation follows directly from the standard reaction-channel (stoichiometric) construction of density-dependent Markov chains and their diffusion (Langevin) approximations
\citep{kurtz_solutions_1970,kurtz_limit_1971,kurtz_limit_1976,ethier_markov_1986,gillespie_exact_1977,gillespie_chemical_2000,anderson_continuous_2011,kampen_stochastic_2007}.
In this framework, the covariance matrix is obtained by summing $\lambda_e(z)\,\Delta_e\Delta_e^\top$ over all reaction events $e$. 
Events that act simultaneously on multiple components therefore generate off-diagonal covariance terms. We illustrate the mechanistic roadmap in Figure~\ref{fig:roadmap}.

A large fraction of the stochastic predator-prey literature introduces randomness by directly perturbing a deterministic ordinary differential equation (ODE). This is most often done by adding Gaussian (or L\'evy) forcing. In practice, the noise typically enters through diagonal multiplicative terms driven by independent Brownian motions, or it acts on a single component only.
For example, stochastic Holling-type II predator--prey models are often written in the form
\[
dN=\cdots\,dt+\sigma_1 N\,dB_1,\quad dP=\cdots\,dt+\sigma_2 P\,dB_2
\]
where $B_1$ and $B_2$ are independent Brownian motions. This choice yields a diffusion matrix that is diagonal at the level of instantaneous covariation \citep{zhang_global_2017,jiang_analysis_2020,zou_survivability_2020,huang_stochastic_2021}.
More generally, reviews and applied studies of noisy consumer-resource cycles often adopt stochastically forced ODEs in which the noise enters through one or several independent exogenous drivers 
\citep{barraquand_moving_2017}.
While such formulations are useful for representing environmental variability, they do not encode the demographic event structure underlying consumer--resource interactions.
As a result, they cannot, by construction, distinguish between noise sources that act independently on each species and those generated by coupled events. 
In particular, they fail to capture predation-with-conversion events that simultaneously decrease prey abundance and increase predator abundance.
Other coupled mechanistic models display Turing patterning, indicating that the absence of coupling structures will result in the loss of critical system behaviors \citep{liu_bidirectional_2025}.  
In studies that emphasize internal stochasticity in predator--prey SDEs, the stochastic model is typically postulated directly at the SDE level. Analysis is then carried out using Lyapunov techniques \citep{abundo_stochastic_1991}. By contrast, these models are not derived from an explicit event-based CTMC scaling.

The present paper closes this modeling gap by starting from a mechanistic CTMC specified in terms of elementary demographic events (Figure~\ref{fig:two_approach}). These events include birth, competition, death, and a predation--conversion channel. We then apply Kurtz's density-dependent scaling to derive a diffusion approximation with explicit drift and covariance structures.
In particular, the coupled predation--conversion event necessarily generates a strictly negative instantaneous cross-covariance, $\Sigma_{12}<0$, throughout the interior of the state space. 
This feature is structural: it is absent from diagonal or independent-noise surrogate models.
Moreover, it follows directly from the reaction-channel covariance construction
$\Sigma(z)=\sum_e \lambda_e(z)\,\Delta_e\Delta_e^\top$ \citep{anderson_constrained_2019,leite_constrained_2019}.
Recent work has developed general SDE-based frameworks for predator–prey models with pure demographic noise. It proves well-posedness and moment/asymptotic bounds and argues that demographic noise alone need not imply extinction \citep{wang_analysis_2025-1}.

\begin{figure}
    \centering
    \includegraphics[width=0.9\linewidth]{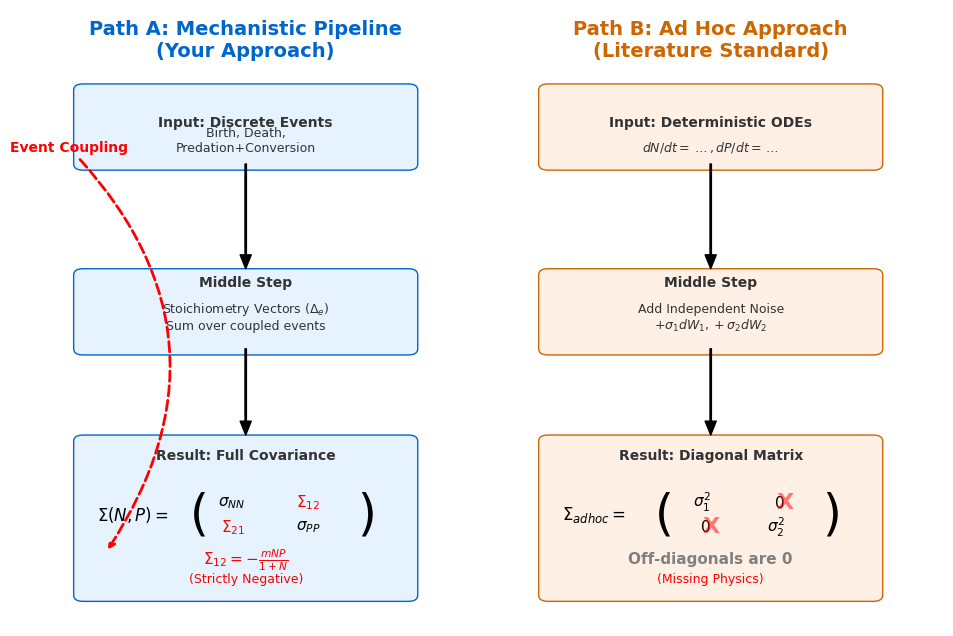}
    \caption{\textbf{Mechanistic vs. Ad Hoc Noise.}
        Path A derives the diffusion matrix from demographic events, revealing the intrinsic negative
        correlation caused by predation. Path B assumes independent noise, missing this structural coupling.}
    \label{fig:two_approach}
\end{figure}

The paper makes the following contributions.

\begin{enumerate}[label=(C\arabic*)]
\item (Mechanistic CTMC and density-dependent scaling.)
We formulate a prey--predator CTMC on $\mathbb{N}_0^2$ from four elementary demographic event channels
(birth, competition, predation--conversion, and predator death) via a random time-change representation.
Under a standard density-dependent scaling in the sense of Kurtz, we obtain the Rosenzweig--MacArthur
ODE as the law-of-large-number (LLN) limit and a diffusion approximation under the central-limit-theorem (CLT)
scaling with noise amplitude $\rho=\Omega^{-1/2}$.

\item (Full covariance and mechanistic negative correlation.)
We compute the limiting drift and the explicit diffusion matrix $\Sigma(N,P)$ implied by the event model.
In particular, the coupled predation--conversion channel yields a strictly negative cross-covariance
$\Sigma_{12}(N,P)=-\frac{mNP}{1+N}<0$ on the interior, providing a direct mechanistic signature that is absent
from diagonal/independent-noise surrogates.
We record two convenient Brownian factorizations: an event-based factorization
$\Sigma=L_{\rm ev}L_{\rm ev}^{\top}$ driven by a $4$-dimensional Brownian motion, and a $2\times2$ Cholesky
factorization on $(0,\infty)^2$.

\item (Absorbed SDE and well-posedness up to extinction.)
We define a full-covariance It\^o diffusion on $U=(0,\infty)^2$ and impose absorption at the axes by
freezing the trajectory upon first hitting $\partial U=\{N=0\}\cup\{P=0\}$.
For the resulting absorbed model we prove strong existence and pathwise uniqueness up to the absorption time,
together with non-explosion and uniform moment bounds on finite time horizons.

\item (Extinction structure.)
We show that extinction occurs with strictly positive probability from every interior initial condition,
for all parameter values. Moreover, in the subcritical regime $m\le c$, predator extinction occurs almost surely,
so the absorbed process hits the predator axis in finite time with probability one.

\item (Reproducible numerical diagnostics.)
We document Euler--Maruyama (EM) simulation procedures for the absorbed diffusion and provide diagnostics that map
directly to the theory: (i) a finite-sample consistency check comparing event-based and Cholesky-based
factorizations (4D vs 2D noise), and (ii) a covariance-role comparison between the full mechanistic model and a
diagonal surrogate with matched marginal variances.
\end{enumerate}

We collect the main results in a form suitable for quick reference.
Throughout, parameters satisfy $k>0$, $m>0$, $c>0$, and the demographic noise amplitude is denoted by $\rho>0$.
Write $U:=(0,\infty)^2$ and $\overline U:=[0,\infty)^2$.

\medskip

\begin{theorem}[Mechanistic diffusion and covariance structure]\label{thm:intro_mech_diffusion}
Consider the mechanistic event model with four channels
(prey birth, prey competition death, predator death, and coupled predation--conversion) and the
density-dependent scaling $Z^\Omega=\Omega^{-1}X^\Omega$ with $\rho=\Omega^{-1/2}$.
Then the limiting drift is the Rosenzweig--MacArthur vector field
\begin{equation}\label{eq:intro_mu}
\mu(N,P)=
\begin{pmatrix}
   N-\dfrac{N^2}{k}-\dfrac{mNP}{1+N} \\[8pt]
   \dfrac{mNP}{1+N}-cP
\end{pmatrix},
\qquad (N,P)\in\mathbb{R}_+^2,
\end{equation}
and the associated diffusion (instantaneous covariance) matrix is
\begin{equation}\label{eq:intro_Sigma}
\Sigma(N,P)=
\begin{pmatrix}
N+\dfrac{N^2}{k}+\dfrac{mNP}{1+N} & -\dfrac{mNP}{1+N}\\[8pt]
-\dfrac{mNP}{1+N} & cP+\dfrac{mNP}{1+N}
\end{pmatrix}.
\end{equation}
In particular, on $U$ one has the strict mechanistic negative correlation
$\Sigma_{12}(N,P)=-\dfrac{mNP}{1+N}<0$.
Moreover, $\Sigma$ admits an event-based Brownian factorization $\Sigma=L_{\rm ev}L_{\rm ev}^{\top}$
with $L_{\rm ev}:\overline U\to\mathbb{R}^{2\times 4}$ given explicitly by the reaction channels, and a
$2\times2$ Cholesky factorization on $U$.
\end{theorem}

\begin{theorem}[Strong well-posedness up to absorption and non-explosion]\label{thm:intro_wellposed_nonexplosion}
Let $z_0\in U$ and let $W$ be a standard $\mathbb{R}^4$ Brownian motion.
Consider the interior SDE on $U$ driven by the event factor $L_{\rm ev}$ from
Theorem~\ref{thm:intro_mech_diffusion}:
\begin{equation}\label{eq:intro_SDE}
dZ(t)=\mu(Z(t))\,dt+\rho\,L_{\rm ev}(Z(t))\,dW(t),\qquad Z(0)=z_0,
\end{equation}
and define the absorption time $\tau:=\inf\{t>0:\,Z(t)\notin U\}=\inf\{t>0:\,N(t)=0\ \text{or}\ P(t)=0\}$.
Then:
\begin{enumerate}
\item[(i)] There exists a unique (up to indistinguishability) strong solution $Z$ to \eqref{eq:intro_SDE} on $[0,\tau)$.
\item[(ii)] The absorbed (frozen) extension $\widehat Z(t):=Z(t)$ for $t<\tau$ and $\widehat Z(t):=Z(\tau)$ for $t\ge\tau$
is a continuous adapted $\overline U$-valued process and $\partial U$ is absorbing for $\widehat Z$.
\item[(iii)] The interior dynamics do not explode prior to absorption: the maximal lifetime equals $\tau$ a.s.
Moreover, for each $T>0$ and $p=2$ or $p\ge 4$ there exists $C_{p,T}<\infty$ such that
\[
\sup_{0\le t\le T}\mathbb{E}\bigl[\,|Z(t\wedge\tau)|^p\,\bigr]\le C_{p,T}\bigl(1+|z_0|^p\bigr).
\]
\end{enumerate}
\end{theorem}

\begin{theorem}[Extinction properties]\label{thm:intro_extinction}
Let $\widehat Z$ be the absorbed diffusion from Theorem~\ref{thm:intro_wellposed_nonexplosion}.
\begin{enumerate}
\item[(i)] (Positive extinction probability for all parameters)
For every $z\in U$,
\[
\mathbb{P}_z(\tau<\infty)>0.
\]
\item[(ii)] (Almost sure predator extinction in the subcritical regime)
If $m\le c$, then for every $z\in U$,
\[
\mathbb{P}_z\bigl(\inf\{t\ge0:\,P(t)=0\}<\infty\bigr)=1,
\]
and consequently $\mathbb{P}_z(\tau<\infty)=1$.
\end{enumerate}
\end{theorem}

\begin{remark}[Roadmap to proofs]
The explicit drift/covariance structure and Brownian factorizations are established in
Section~\ref{sec:mu_sigma_factorization}.
Strong well-posedness up to absorption and non-explosion/moment bounds is proved in
Section~\ref{sec:sde_absorption_wellposedness}.
The extinction assertions in Theorem~\ref{thm:intro_extinction} are proved in
Section~\ref{sec:boundary_and_extinction_A}, with details for the extinction in the subcritical regime deferred to Appendix~\ref{app:cir_hitting} and the positive-probability argument deferred to
Appendix~\ref{app:positive_extinction_probability}.
\end{remark}

Section~\ref{sec:deterministic_backbone_A} reviews the deterministic Rosenzweig--MacArthur mean-field backbone needed for parameter interpretation and local stability properties. Section~\ref{sec:mech_ctmc} specifies the mechanistic CTMC through elementary demographic events. Sections~\ref{sec:scaling_diffusion}--\ref{sec:mu_sigma_factorization} introduce the density-dependent scaling,
derive the diffusion approximation by identifying the limiting drift and covariance, and record both the event-based ($4$D) and Cholesky ($2$D) Brownian factorizations. Section~\ref{sec:sde_absorption_wellposedness} introduces the absorbed (frozen) full-covariance SDE and proves strong well-posedness up to the absorption time, together with non-explosion and moment bounds. 
Section~\ref{sec:boundary_and_extinction_A} studies the boundary structure
and establishes extinction results, including a positive probability of extinction for all parameters and almost sure predator extinction in the subcritical regime. Section~\ref{sec:numerics} documents numerical schemes and diagnostics supporting the theoretical statements, with emphasis on factorization consistency and the role of the full covariance. We conclude this work in the Discussion Section.

\section{Deterministic mean-field backbone}
\label{sec:deterministic_backbone_A}

This section records the deterministic Rosenzweig--MacArthur ODE that arises as the LLN
limit of the mechanistic CTMC under Kurtz scaling (Section~\ref{sec:scaling_diffusion}). 
Our purpose here is purely
auxiliary. We fix the notation and biologically meaningful state space and record basic invariance and
boundedness properties (dissipativity). We also summarize deterministic coexistence versus extinction
thresholds for parameter interpretation. We do not pursue the Hopf bifurcation or global planar dynamics in this
paper.

We assume
\begin{equation}\label{eq:det_params_assumptions_A}
k>0,\qquad m>0,\qquad c>0,
\end{equation}
and interpret $N(t)\ge0$ as prey density and $P(t)\ge0$ as predator density.
The biologically meaningful state space is the closed quadrant
\[
\overline U := \{(N,P)\in\mathbb R^2:\ N\ge 0,\ P\ge 0\},\qquad
U := \{(N,P)\in\mathbb R^2:\ N>0,\ P>0\}.
\]

For an initial condition $(N_0,P_0)\in \overline U$ we consider
\begin{equation}\label{eq:RM_ODE}
\left\{
\begin{aligned}
\dot N &= N\Bigl(1-\frac{N}{k}\Bigr)-\frac{mNP}{1+N},\\[1mm]
\dot P &= P\Bigl(\frac{mN}{1+N}-c\Bigr),
\end{aligned}\right.
\qquad (N(0),P(0))=(N_0,P_0).
\end{equation}
Let $F:\overline U\to\mathbb R^2$ denote the vector field in \eqref{eq:RM_ODE}. Note that $F$ is $C^\infty$ on $U$
and locally Lipschitz on $\overline U$.

\begin{lemma}[Local well-posedness]\label{lem:det_local_wellposed_A}
For every $(N_0,P_0)\in \overline U$, there exists a maximal time $T_{\max}\in(0,\infty]$ and a unique maximal
solution
\[
(N,P)\in C^1\bigl([0,T_{\max});\overline U\bigr)
\]
to \eqref{eq:RM_ODE}. Moreover, the solution depends continuously on the initial data on compact time intervals.
\end{lemma}

\begin{proof}
Since $F$ is locally Lipschitz on $\overline U$, existence and uniqueness of a maximal solution follow from the
Picard--Lindel\"of theorem; continuous dependence is standard.
\end{proof}

\begin{lemma}[Positive invariance of $\overline U$]\label{lem:det_positive_invariance_A}
The closed quadrant $\overline U$ is forward invariant for \eqref{eq:RM_ODE}.
In addition, the coordinate axes are invariant.
\end{lemma}

\begin{proof}
On $\overline U$ the predator equation reads $\dot P = P\,\psi(N)$ with
$\psi(N):=\frac{mN}{1+N}-c$, hence
\[
P(t)=P_0\exp\Bigl(\int_0^t \psi(N(s))\,ds\Bigr)\ge 0,
\qquad t\in[0,T_{\max}).
\]
In particular, if $P_0=0$ then $P(t)\equiv 0$.
Similarly, writing $\dot N = N\,\phi(N,P)$ with
\[
\phi(N,P):=\Bigl(1-\frac{N}{k}\Bigr)-\frac{mP}{1+N},
\]
we obtain
\[
N(t)=N_0\exp\Bigl(\int_0^t \phi(N(s),P(s))\,ds\Bigr)\ge 0,
\qquad t\in[0,T_{\max}).
\]
If $N_0=0$ then $N(t)\equiv 0$, which forces $P(t)\to 0$ since $\dot P=-cP$ on $\{N=0\}$.
\end{proof}


We establish dissipativity of \eqref{eq:RM_ODE} by combining a logistic comparison for $N$ with a linear Lyapunov
combination controlling $N+\beta P$. This yields a compact absorbing set in $\overline U$ that is independent of the
initial condition, and in particular implies global existence.

\begin{lemma}[Logistic comparison for the prey]\label{lem:prey_logistic_bound_A}
Let $(N,P)$ be the maximal solution to \eqref{eq:RM_ODE} with $(N_0,P_0)\in\overline U$. Then
\begin{equation}\label{eq:prey_bound_pointwise_A}
0\le N(t)\le \max\{N_0,k\}\qquad \text{for all }t\in[0,T_{\max}).
\end{equation}
Moreover, for every $\delta>0$ there exists $T_\delta=T_\delta(N_0)\ge 0$ such that
\begin{equation}\label{eq:prey_eventual_bound_A}
0\le N(t)\le k+\delta\qquad \text{for all }t\ge T_\delta \text{ with }t<T_{\max}.
\end{equation}
\end{lemma}

\begin{proof}
Since $P(t)\ge 0$,
\[
\dot N(t)=N(t)\Bigl(1-\frac{N(t)}{k}\Bigr)-\frac{mN(t)P(t)}{1+N(t)}
\le N(t)\Bigl(1-\frac{N(t)}{k}\Bigr).
\]
Comparison with $\dot y=y(1-y/k)$ yields \eqref{eq:prey_bound_pointwise_A}. The eventual bound
\eqref{eq:prey_eventual_bound_A} follows from the explicit logistic solution: for any $\delta>0$ the logistic
trajectory enters and stays in $[0,k+\delta]$ after some finite time depending only on $N_0$.
\end{proof}

For $\beta\in(0,1]$, define the linear functional
\begin{equation}\label{eq:Sbeta_def_A}
S_\beta(N,P):=N+\beta P.
\end{equation}

\begin{lemma}[Absorbing half-space via $S_\beta$]\label{lem:absorbing_halfspace_A}
Fix $\beta\in(0,1]$ and let $(N,P)$ solve \eqref{eq:RM_ODE}. Then, for all $t\in[0,T_{\max})$,
\begin{equation}\label{eq:Sbeta_diff_ineq_A}
\frac{d}{dt}S_\beta(N(t),P(t))
\le \frac{k}{4}(1+c)^2 - c\,S_\beta(N(t),P(t)).
\end{equation}
Consequently,
\begin{equation}\label{eq:Sbeta_bound_A}
S_\beta(N(t),P(t))
\le S_\beta(N_0,P_0)e^{-ct}+\frac{k(1+c)^2}{4c}\bigl(1-e^{-ct}\bigr),
\qquad t\in[0,T_{\max}),
\end{equation}
and for every $\varepsilon>0$ there exists $T_\varepsilon=T_\varepsilon(N_0,P_0)\ge 0$ such that
\begin{equation}\label{eq:Sbeta_absorption_A}
S_\beta(N(t),P(t))\le \frac{k(1+c)^2}{4c}+\varepsilon
\qquad \text{for all }t\ge T_\varepsilon \text{ with }t<T_{\max}.
\end{equation}
\end{lemma}

\begin{proof}
Differentiate $S_\beta$ along \eqref{eq:RM_ODE}:
\[
\frac{d}{dt}S_\beta
=\dot N+\beta \dot P
= N\Bigl(1-\frac{N}{k}\Bigr)-\frac{mNP}{1+N}+\beta P\Bigl(\frac{mN}{1+N}-c\Bigr).
\]
Rearranging,
\begin{align*}
\frac{d}{dt}S_\beta
&= N\Bigl(1-\frac{N}{k}\Bigr)-\beta cP-(1-\beta)\frac{mNP}{1+N}
\le N\Bigl(1-\frac{N}{k}\Bigr)-\beta cP \\[4pt]
&\le N\Bigl(1-\frac{N}{k}\Bigr)+cN-cS_\beta
=(1+c)N-\frac{N^2}{k}-cS_\beta.
\end{align*}
Maximizing the concave quadratic in $N\ge 0$ yields
\[
(1+c)N-\frac{N^2}{k}\le \frac{k}{4}(1+c)^2,
\]
which gives \eqref{eq:Sbeta_diff_ineq_A}. Solving the scalar differential inequality yields \eqref{eq:Sbeta_bound_A},
and \eqref{eq:Sbeta_absorption_A} follows.
\end{proof}

For $\delta,\varepsilon>0$ and $\beta\in(0,1]$, define the compact set
\begin{equation}\label{eq:absorbing_set_def_A}
\mathcal A_{\beta,\varepsilon,\delta}
:=\Bigl\{(N,P)\in\overline U:\ 0\le N\le k+\delta,\ \ 0\le N+\beta P\le \frac{k(1+c)^2}{4c}+\varepsilon\Bigr\}.
\end{equation}

\begin{theorem}[Dissipativity, absorbing set, and global existence]\label{thm:absorbing_set_global_A}
Fix $\beta\in(0,1]$ and $\varepsilon,\delta>0$. For every initial condition $(N_0,P_0)\in\overline U$, the
corresponding solution of \eqref{eq:RM_ODE} is global, i.e.\ $T_{\max}=\infty$, and there exists a finite time
$T=T(N_0,P_0;\beta,\varepsilon,\delta)\ge 0$ such that
\[
(N(t),P(t))\in \mathcal A_{\beta,\varepsilon,\delta}
\qquad\text{for all }t\ge T.
\]
In particular, \eqref{eq:RM_ODE} is dissipative on $\overline U$ and admits a uniform absorbing set
(independent of the initial condition).
\end{theorem}

\begin{proof}
Lemma~\ref{lem:prey_logistic_bound_A} yields that for the given $\delta>0$ there exists $T_1$ such that
$N(t)\le k+\delta$ for all $t\ge T_1$.
Lemma~\ref{lem:absorbing_halfspace_A} yields that for the given $\varepsilon>0$ there exists $T_2$ such that
$N(t)+\beta P(t)\le \frac{k(1+c)^2}{4c}+\varepsilon$ for all $t\ge T_2$.
Thus, with $T:=\max\{T_1,T_2\}$ we have $(N(t),P(t))\in \mathcal A_{\beta,\varepsilon,\delta}$ for all $t\ge T$.

Finally, since $F$ is locally Lipschitz and solutions remain in the compact set $\mathcal A_{\beta,\varepsilon,\delta}$
for all sufficiently large times, no finite-time blow-up can occur. Hence $T_{\max}=\infty$.
\end{proof}

Figure~\ref{fig:absorbing_set} shows the absorbing set, providing a visual proof for dissipativity.

\begin{figure}[htbp]
    \centering
    \includegraphics[width=0.9\linewidth]{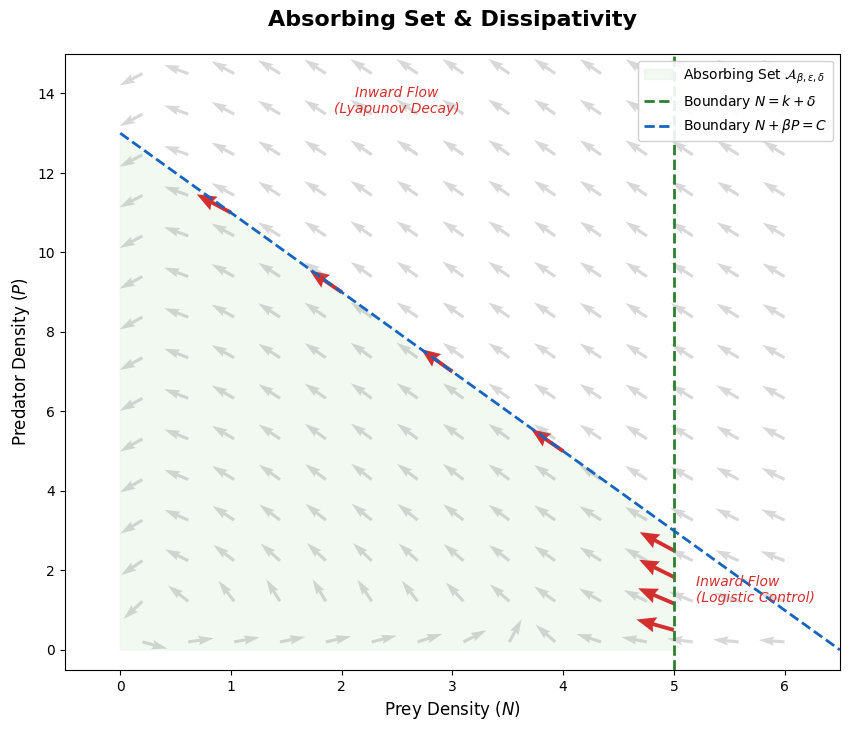}
\caption{\textbf{Visual proof of dissipativity and the absorbing set.} 
Illustration of the compact invariant region $\mathcal{A}_{\beta,\varepsilon,\delta}$ defined in Theorem~\ref{thm:absorbing_set_global_A}. The region (light green) is bounded by the prey logistic constraint $N=k+\delta$ (green dashed line) and the linear Lyapunov constraint $N+\beta P=C$ (blue dashed line). The red arrows along the boundaries explicitly show the inward direction of the vector field. They demonstrate that any trajectory starting outside will eventually enter and remain within this set, thereby ensuring global existence and non-explosion for the deterministic backbone.}
\label{fig:absorbing_set}
\end{figure}

We summarize equilibria and deterministic extinction/coexistence thresholds for later parameter interpretation.
Define, when $m>c$,
\begin{equation}\label{eq:Nstar_def_A}
N^\ast := \frac{c}{m-c}.
\end{equation}

\begin{proposition}[Equilibria]\label{prop:equilibria_A}
The equilibria of \eqref{eq:RM_ODE} in $\overline U$ are:
\begin{enumerate}
\item The origin $K_0=(0,0)$.
\item The prey-only equilibrium $K_1=(k,0)$.
\item If $m>c$ and $N^\ast<k$, there exists a unique coexistence equilibrium
\[
K_3=(N^\ast,P^\ast),\qquad 
P^\ast := \frac{1+N^\ast}{m}\Bigl(1-\frac{N^\ast}{k}\Bigr)>0.
\]
If $m\le c$ or $m>c$ with $N^\ast\ge k$, then no positive equilibrium exists.
\end{enumerate}
\end{proposition}

\begin{proof}
Equilibria satisfy $\dot N=\dot P=0$. From $\dot P=P(\frac{mN}{1+N}-c)$ we obtain $P=0$ or
$\frac{mN}{1+N}=c$, the latter equivalent to $N=N^\ast$ when $m>c$. If $P=0$, then $\dot N=N(1-N/k)$ gives $N\in\{0,k\}$.
If $m>c$ and $N=N^\ast$, then $\dot N=0$ yields $P=h(N^\ast)$ with $h(N)=\frac{1+N}{m}(1-\frac{N}{k})$; positivity
requires $N^\ast<k$.
\end{proof}

Write the predator per-capita growth rate as
\begin{equation}\label{eq:psi_def_A}
\psi(N):=\frac{mN}{1+N}-c,
\qquad\text{so that}\qquad 
\dot P=\psi(N)\,P.
\end{equation}

\begin{theorem}[Deterministic predator--prey dynamics]\citep{grunert_evolutionarily_2021,cheng_uniqueness_1981}\label{thm:det_predator_extinction}
Let $(N(t),P(t))$ solve \eqref{eq:RM_ODE} with $N_0\in (0,K)$ and $P_0\in (0,\infty)$. Then we have the following system dynamics:
\begin{enumerate}[label=(\roman*)]
\item If $m\le c$ or $k\le N^\ast$, then the critical point $K_1$ is asymptotically stable and 
\[
\lim_{t\to \infty}\,N(t) = K,\quad \lim_{t\to \infty}\,P(t) = 0.
\]
\item If $N^\ast<K\le 1 + 2N^\ast$, then the critical point $K_3$ is asymptotically stable and 
\[
\lim_{t\to \infty}\,N(t) = N^\ast,\quad \lim_{t\to \infty}\,P(t) = P^\ast.
\]
\item If $K> 1+2N^\ast$, then $K_3$ is unstable and there exists exactly one periodic orbit in the first quadrant in the $(N,P)$ plane, which is an (asymptotically) stable limit cycle.
\end{enumerate}
\end{theorem}

Figure~\ref{fig:dynamic_regimes} shows three dynamic regimes of the deterministic ODE \eqref{eq:RM_ODE}.

\begin{figure}[htbp]
    \centering
    \includegraphics[width=0.75\linewidth]{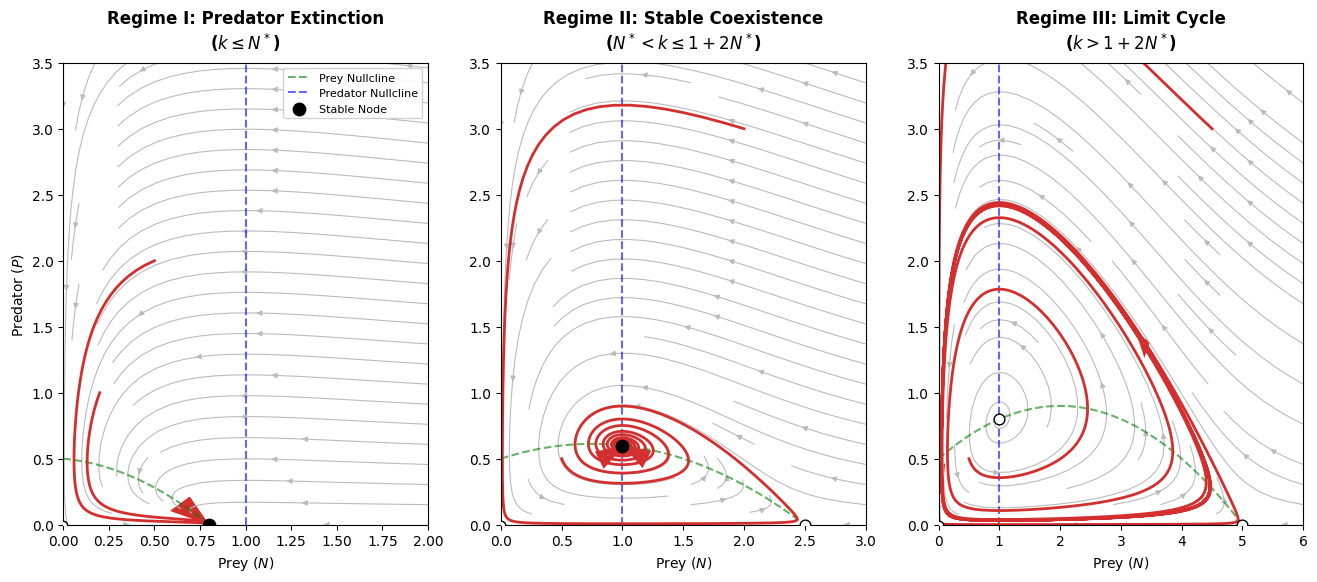}
\caption{\textbf{The three dynamic regimes of the deterministic backbone.}
Numerical illustration of the stability thresholds detailed in Theorem~\ref{thm:det_predator_extinction}.
\textbf{(Left) Regime I:} Low carrying capacity ($k \le N^*$) leads to predator extinction; trajectories converge to $(k,0)$.
\textbf{(Center) Regime II:} Intermediate carrying capacity ($N^* < k \le 1+2N^*$) ensures stable coexistence; trajectories spiral into the positive equilibrium $K_3$.
\textbf{(Right) Regime III:} High carrying capacity ($k > 1+2N^*$) destabilizes $K_3$ via a Hopf bifurcation, resulting in a stable limit cycle. This transition illustrates the paradox of Enrichment, where increasing resource availability (larger $k$) induces large-amplitude oscillations.}
\label{fig:dynamic_regimes}
\end{figure}

\begin{remark}[Dynamics beyond the present scope]\label{rmk:hopf_bifurcation_dynamics}
The quantity $1+2N^\ast$ is the Hopf bifurcation threshold for $k$, where the system undergoes bifurcation. When $k$ increases and exceeds $1+2N^*$, the interior equilibrium $K_3$ loses stability and an asymptotically stable limit cycle emerges, where the predator--prey system manifests as periodically oscillating population cycles. This bifurcation phenomenon facilitates the deterministic Rosenzweig–MacArthur predator–prey model to explain the ``paradox of enrichment'' \citep{gounand_paradox_2014}. Related algebraic–spectral threshold analyses and explicit Hopf onset criteria for Leslie–Gower-type systems have recently been developed, including discrete–continuous stability and bifurcation results \citep{wang_algebraicspectral_2026}.  
\end{remark}

\section{Mechanistic event model: a CTMC with absorbing axes}
\label{sec:mech_ctmc}

We introduce a mechanistic demographic model at the jump-process level by specifying a CTMC for prey--predator counts. The model is constructed from elementary demographic events (birth,
competition, predation/conversion, and death) with state-dependent intensities. This event setting is the input
for the system-size scaling and diffusion approximation developed in
Sections~\ref{sec:scaling_diffusion}--\ref{sec:mu_sigma_factorization}.

\subsection{State space, event channels, increments, and intensities}
\label{subsec:ctmc_specification_A}

Let
\[
X(t)=(\mathcal N(t),\mathcal P(t))^\top\in\mathbb{N}_0^2
\]
denote prey and predator counts at time $t\ge 0$. Conditional on the current state $(n,p)\in\mathbb{N}_0^2$, the
process experiences one of finitely many elementary events $e\in\mathcal E=\{B,C,D,E\}$. Each event has a fixed
increment $\Delta_e\in\mathbb{Z}^2$ and an intensity $\Lambda_e(n,p)\ge 0$.

The four event channels are:
\begin{center}
\begin{tabular}{llll}
\hline
Event $e$ & Interpretation & Increment $\Delta_e$ & Intensity $\Lambda_e(n,p)$\\
\hline
$B$ & prey birth & $(1,0)^\top$ & $n$\\
$C$ & prey competition death & $(-1,0)^\top$ & $\dfrac{n^2}{k}$\\[1ex]
$D$ & predator death & $(0,-1)^\top$ & $c\,p$\\
$E$ & predation + conversion & $(-1,1)^\top$ & $\dfrac{mnp}{1+n}$\\[1ex]
\hline
\end{tabular}
\end{center}
Here $k>0$ is the prey carrying-capacity parameter, $m>0$ is the maximal predation/conversion rate, and $c>0$ is the
predator mortality rate. The coupled predation/conversion event $E$ changes both coordinates simultaneously and is the
source of the negative prey--predator noise covariance in the diffusion approximation
(Section~\ref{sec:mu_sigma_factorization}).

For random time-change representation, let $(Y_e)_{e\in\mathcal E}$ be independent unit-rate Poisson processes. Define the event counting processes
\[
N_e(t):=Y_e\!\left(\int_0^t \Lambda_e\bigl(X(s)\bigr)\,ds\right),
\qquad e\in\mathcal E,
\]
so that $N_e$ has predictable intensity $\Lambda_e(X(t))$ with respect to the natural filtration. Then the CTMC admits
the exact stochastic integral representation
\begin{equation}\label{eq:ctmc_jump_representation_A}
X(t)=X(0)+\sum_{e\in\mathcal E}\Delta_e\,N_e(t),
\qquad t\ge 0,
\end{equation}
equivalently, in differential form,
\begin{equation}\label{eq:ctmc_differential_form_A}
dX(t)=\sum_{e\in\mathcal E}\Delta_e\,dN_e(t).
\end{equation}
Writing components,
\begin{equation}\label{eq:ctmc_components_A}
\begin{aligned}
d\mathcal N(t) &= dN_B(t)-dN_C(t)-dN_E(t),\\[4pt]
d\mathcal P(t) &= dN_E(t)-dN_D(t).
\end{aligned}
\end{equation}

\subsection{Absorbing axes (demographic extinction)}
\label{subsec:ctmc_absorbing_axes}

The coordinate axes are absorbing for the CTMC. Indeed, if $\mathcal N(t)=0$ then
$\Lambda_B=\Lambda_C=\Lambda_E=0$, so $\mathcal N(s)=0$ for all $s\ge t$. Similarly, if $\mathcal P(t)=0$ then
$\Lambda_D=\Lambda_E=0$, so $\mathcal P(s)=0$ for all $s\ge t$. Thus demographic extinction at the count level is
irreversible.

The event formulation above will be carried through a density-dependent scaling to obtain a diffusion approximation.
Because the predation--conversion event $E$ changes prey and predator in opposite directions within the same
reaction channel, it induces an instantaneous negative cross-covariance in the diffusion matrix
$\Sigma=\sum_e \lambda_e\,\Delta_e\Delta_e^\top$ (Section~\ref{sec:mu_sigma_factorization}).

\section{System-size scaling and diffusion approximation}
\label{sec:scaling_diffusion}

The CTMC in Section~\ref{sec:mech_ctmc} is formulated at the level of integer counts and therefore carries a natural system-size parameter. 
Two diffusion-based approximations are commonly employed for Markov chain models: the linear noise approximation \citep{kampen_power_1961,kampen_stochastic_2007} and the chemical Langevin equation \citep{gillespie_chemical_2000,gillespie_chemical_2002,kurtz_strong_1978}. Linearizing stochastic fluctuations around the deterministic limit derive the linear noise approximation. While this approximation is mathematically well defined for all times, it often evolves outside the positive orthant and may therefore yield nonphysical negative concentration values. Moreover, it is well known that this approach can inadequately represent fluctuations arising from nonlinear reaction rates \citep{wallace_linear_2012}. By contrast, the chemical Langevin equation typically provides a more accurate description in the presence of nonlinearities. However, it is generally only defined up to the first time the process reaches the boundary of the positive orthant. Indeed, because the diffusion coefficients usually contain square roots of molecular concentrations, the unstopped equation becomes ill-posed beyond this point \citep{manninen_developing_2006,szpruch_comparing_2010,wilkie_positivity_2008}.

In this work, we adopt the chemical Langevin approach and focus on the system evolution up to the first hitting time of the domain boundary. Our interpretation of the evolution beyond the first hitting time is freezing the trajectory upon the hitting time, i.e., the trajectory remains where it hits the boundary thereafter. To obtain a diffusion approximation with an explicit noise amplitude and a transparent
deterministic limit, we introduce a scaling parameter $\Omega\ge 1$ (effective population size/volume) and work with density variables. 
The resulting limit falls under the classical density-dependent Markov chain framework of Kurtz \citep{kurtz_solutions_1970,kurtz_limit_1971,kurtz_limit_1976,ethier_markov_1986}.


For each $\Omega\ge 1$, let $X^\Omega(t)=(\mathcal N^\Omega(t),\mathcal P^\Omega(t))^\top$ be a CTMC on
$\mathbb{N}_0^2$ with the same event increments $\Delta_e$ as in Section~\ref{subsec:ctmc_specification_A}.
We define the density (rescaled) process
\begin{equation}\label{eq:density_definition_A}
Z^\Omega(t)=\bigl(N^\Omega(t),P^\Omega(t)\bigr)^\top
:=\Omega^{-1}X^\Omega(t)\in \Omega^{-1}\mathbb{N}_0^2\subset\mathbb{R}_+^2.
\end{equation}
Under this scaling, each event changes the density by $\Omega^{-1}\Delta_e$, while the expected number of events over
$O(1)$ time intervals is $O(\Omega)$.

For Density-dependent intensities, we choose $\Omega$-dependent intensities of density-dependent form
\begin{equation}\label{eq:intensities_scaled_A}
\Lambda_e^\Omega(x)=\Omega\,\lambda_e(\Omega^{-1}x),
\qquad x\in\mathbb{N}_0^2,\quad e\in\mathcal E,
\end{equation}
where the rate functions $\lambda_e:\mathbb{R}_+^2\to\mathbb{R}_+$ are
\begin{equation}\label{eq:rate_functions_A}
\lambda_B(N,P)=N,\qquad
\lambda_C(N,P)=\frac{N^2}{k},\qquad
\lambda_D(N,P)=cP,\qquad
\lambda_E(N,P)=\frac{mNP}{1+N}.
\end{equation}
With this choice, the random time-change representation in \eqref{eq:ctmc_jump_representation_A} yields
\begin{equation}\label{eq:density_jump_representation_A}
Z^\Omega(t)=Z^\Omega(0)+\sum_{e\in\mathcal E}\frac{\Delta_e}{\Omega}\,
Y_e\!\left(\Omega\int_0^t \lambda_e\bigl(Z^\Omega(s)\bigr)\,ds\right),
\end{equation}
where $(Y_e)_{e\in\mathcal E}$ are independent unit-rate Poisson processes.

For drift--martingale decomposition and the noise scale. Write $\widetilde Y_e(u)=Y_e(u)-u$ for the compensated Poisson martingales. Substituting
$Y_e(u)=u+\widetilde Y_e(u)$ into \eqref{eq:density_jump_representation_A} gives the semimartingale decomposition
\begin{equation}\label{eq:density_semimartingale_A}
Z^\Omega(t)
=
Z^\Omega(0)
+\int_0^t \mu\bigl(Z^\Omega(s)\bigr)\,ds
+\rho\,M^\Omega(t),
\end{equation}
where the deterministic drift is
\begin{equation}\label{eq:mu_definition_A}
\mu(z)=\sum_{e\in\mathcal E}\Delta_e\,\lambda_e(z),
\qquad z\in\mathbb{R}_+^2,
\end{equation}
and we introduce the canonical demographic noise amplitude
\begin{equation}\label{eq:rho_definition_A}
\rho=\Omega^{-1/2}.
\end{equation}
The fluctuation term $M^\Omega$ is an $\mathbb{R}^2$-valued martingale defined by
\begin{equation}\label{eq:martingale_MOmega_A}
M^\Omega(t)
:=
\frac{1}{\sqrt{\Omega}}\sum_{e\in\mathcal  E} \Delta_e\,
\widetilde Y_e\!\left(\Omega\int_0^t \lambda_e(Z^\Omega(s))\,ds\right).
\end{equation}
Its predictable quadratic variation is
\begin{equation}\label{eq:quad_var_MOmega_A}
\left\langle M^\Omega\right\rangle(t)
=
\int_0^t \Sigma\bigl(Z^\Omega(s)\bigr)\,ds,
\qquad
\Sigma(z):=\sum_{e\in\mathcal E}\lambda_e(z)\,\Delta_e\Delta_e^\top.
\end{equation}
Thus $\Sigma$ is $O(1)$ while the prefactor $\rho=\Omega^{-1/2}$ in \eqref{eq:density_semimartingale_A} exhibits the
$\Omega^{-1/2}$ fluctuation scale.

For LLN and diffusion limit (informal statement), the following is the canonical limiting picture for density-dependent Markov chains. We state it informally since it is
standard; precise hypotheses and proofs can be found in \citep{kurtz_solutions_1970,kurtz_limit_1971,kurtz_limit_1976,ethier_markov_1986}. The general theme that links stochastic microscale rules to deterministic macroscopic limits also appears in hybrid PDE–ABM settings. Under these settings, mean-field PDE limits are derived from Gillespie-driven agent dynamics \citep{wang_analysis_2025}.  

\begin{proposition}[LLN/CLT limits for density-dependent scaling (informal)]\label{prop:lln_clt_informal_A}
Let $Z^\Omega$ be defined by \eqref{eq:density_definition_A}--\eqref{eq:density_jump_representation_A} with rate
functions \eqref{eq:rate_functions_A}. Assume $Z^\Omega(0)\to z_0\in\mathbb{R}_+^2$ as $\Omega\to\infty$.
\begin{enumerate}
\item[(LLN)] $Z^\Omega$ converges in probability, uniformly on compact time intervals, to the unique solution $z(\cdot)$
of the mean-field ODE
\begin{equation}\label{eq:LLN_limit_ODE_A}
\dot z(t)=\mu\bigl(z(t)\bigr),\qquad z(0)=z_0.
\end{equation}
\item[(CLT)] The fluctuation process $\sqrt{\Omega}\,(Z^\Omega-z)$ converges in distribution (in the Skorokhod topology)
to an It\^o diffusion whose drift is given by the linearization of $\mu$ along $z(t)$ and whose instantaneous covariance
is $\Sigma(z(t))$.
\end{enumerate}
\end{proposition}

\begin{remark}[Connection to chemical Langevin and reaction-channel form]\label{rem:cle_connection_A}
The covariance representation \eqref{eq:quad_var_MOmega_A} coincides with the standard reaction-channel covariance
formula used in chemical Langevin approximations and stochastic reaction network theory
\citep{gillespie_exact_1977,gillespie_chemical_2000,kampen_stochastic_2007,anderson_continuous_2011}.
In particular, coupled events that simultaneously change multiple coordinates (here, predation--conversion) induce
off-diagonal covariance terms in the diffusion limit.
\end{remark}

\section{Limiting drift and covariance; mechanistic negative correlation}
\label{sec:mu_sigma_factorization}

In this section we compute the limiting drift and instantaneous covariance implied by the mechanistic event model
of Section~\ref{sec:mech_ctmc} under the density-dependent scaling of Section~\ref{sec:scaling_diffusion}.
The off-diagonal entry $\Sigma_{12}<0$ will emerge as a direct mechanistic fingerprint of the coupled
predation--conversion channel: a single event decreases prey while (with conversion) increasing predators.

Explicit drift and diffusion matrix. Let $\lambda_e$ be the density-level event rates in \eqref{eq:rate_functions_A} and let $\Delta_e\in\mathbb{Z}^2$ denote
the event increments from Section~\ref{subsec:ctmc_specification_A}. Recall the definitions
\begin{equation}\label{eq:mu_Sigma_definitions_A}
\mu(z)=\sum_{e\in\mathcal E}\Delta_e\,\lambda_e(z),
\qquad
\Sigma(z)=\sum_{e\in\mathcal E}\lambda_e(z)\,\Delta_e\Delta_e^\top,
\qquad z=(N,P)\in\mathbb{R}_+^2.
\end{equation}
With the conventions
\[
\Delta_B=\binom{1}{0},\qquad
\Delta_C=\binom{-1}{0},\qquad
\Delta_D=\binom{0}{-1},\qquad
\Delta_E=\binom{-1}{1},
\]
and with the rates \eqref{eq:rate_functions_A}, we obtain the explicit drift
\begin{equation}\label{eq:mu_explicit}
\mu(N,P)=
\begin{pmatrix}
N-\dfrac{N^2}{k}-\dfrac{mNP}{1+N} \\[6pt]
\dfrac{mNP}{1+N}-cP
\end{pmatrix}
\end{equation}
which coincides with the Rosenzweig--MacArthur vector field in \eqref{eq:RM_ODE}. Moreover, the diffusion
(instantaneous covariance) matrix is
\begin{equation}\label{eq:Sigma_explicit}
\Sigma(N,P)=
\begin{pmatrix}
N+\dfrac{N^2}{k}+\dfrac{mNP}{1+N} & -\dfrac{mNP}{1+N}\\[6pt]
-\dfrac{mNP}{1+N} & cP+\dfrac{mNP}{1+N}
\end{pmatrix}.
\end{equation}
In particular, for $(N,P)\in U=(0,\infty)^2$ one has the strict mechanistic negative correlation
\begin{equation}\label{eq:Sigma12_negative_A}
\Sigma_{12}(N,P)=-\frac{mNP}{1+N}<0.
\end{equation}

Event-based Brownian factorization, it is often convenient to encode \eqref{eq:Sigma_explicit} by an ``event matrix'' that preserves the
mechanistic decomposition by reaction channels. Define
\begin{equation}\label{eq:Levent_def_A}
L_{\rm ev}(N,P)
:=
\Bigl[\,
\sqrt{\lambda_B(N,P)}\,\Delta_B,\;
\sqrt{\lambda_C(N,P)}\,\Delta_C,\;
\sqrt{\lambda_D(N,P)}\,\Delta_D,\;
\sqrt{\lambda_E(N,P)}\,\Delta_E
\Bigr]\in\mathbb{R}^{2\times 4},
\end{equation}
where $\lambda_e$ are as in \eqref{eq:rate_functions_A}. Writing columns explicitly,
\begin{equation}\label{eq:Levent_explicit_A}
L_{\rm ev}(N,P)=
\begin{pmatrix}
\sqrt{N} & -\sqrt{\dfrac{N^2}{k}} & 0 & -\sqrt{\dfrac{mNP}{1+N}}\\[8pt]
0 & 0 & -\sqrt{cP} & \sqrt{\dfrac{mNP}{1+N}}
\end{pmatrix}.
\end{equation}

\begin{lemma}[Event factorization]\label{lem:event_factorization_A}
For all $(N,P)\in\mathbb{R}_+^2$,
\begin{equation}\label{eq:event_factorization_identity_A}
\Sigma(N,P)=L_{\rm ev}(N,P)\,L_{\rm ev}(N,P)^\top .
\end{equation}
Consequently, if $W(t)=(W_B(t),W_C(t),W_D(t),W_E(t))^\top$ is a standard $\mathbb{R}^4$ Brownian motion, then the
diffusion approximation associated with \eqref{eq:mu_explicit}--\eqref{eq:Sigma_explicit} can be written in
reaction-channel (event) form as
\begin{equation}\label{eq:SDE_event_form}
dZ(t)=\mu(Z(t))\,dt+\rho\,L_{\rm ev}(Z(t))\,dW(t),
\qquad \rho=\Omega^{-1/2},
\end{equation}
whenever the path remains in $U=(0,\infty)^2$.
\end{lemma}

\begin{proof}
By definition \eqref{eq:Levent_def_A},
\begin{align*}
L_{\rm ev}(N,P)L_{\rm ev}(N,P)^\top
=&\sum_{e\in\mathcal E}
\bigl(\sqrt{\lambda_e(N,P)}\,\Delta_e\bigr)\bigl(\sqrt{\lambda_e(N,P)}\,\Delta_e\bigr)^\top \\[4pt]
=&\sum_{e\in \mathcal E}\lambda_e(N,P)\,\Delta_e\Delta_e^\top
=\Sigma(N,P),
\end{align*}
which is \eqref{eq:event_factorization_identity_A}. The SDE representation \eqref{eq:SDE_event_form} follows from the
standard identification of a diffusion with covariance $\Sigma$ via any factor $L$ such that $\Sigma=LL^\top$.
\end{proof}

The factor $L_{\rm ev}$ preserves the mechanistic decomposition by event channels and is therefore natural for both
analysis and simulation. Since there are four reaction channels, $L_{\rm ev}$ is $2\times4$ and the driving Brownian
motion in \eqref{eq:SDE_event_form} is $4$-dimensional. In Section~\ref{subsec:cholesky_factorization_A} we record a
two-dimensional (Cholesky) factorization on $U$ that is sometimes preferable for computation.

Positive definiteness on the interior and degeneracy on the boundary. The diffusion matrix \eqref{eq:Sigma_explicit} is uniformly nondegenerate on compact subsets of the interior $U$,
while it degenerates on the axes where one or more event rates vanish.

\begin{lemma}[Positive definiteness on the interior; boundary degeneracy]\label{lem:Sigma_posdef_A}
For all $(N,P)\in(0,\infty)^2$, the matrix $\Sigma(N,P)$ in \eqref{eq:Sigma_explicit} is symmetric and positive
definite. On the boundary $\partial U=\{N=0\}\cup\{P=0\}$ it is positive semidefinite and may be singular (degenerate).
\end{lemma}

\begin{proof}
Symmetry is immediate from \eqref{eq:Sigma_explicit}. For $N>0$ and $P>0$ we have
\[
\Sigma_{11}(N,P)=N+\frac{N^2}{k}+\frac{mNP}{1+N}>0,
\qquad
\Sigma_{22}(N,P)=cP+\frac{mNP}{1+N}>0.
\]
Moreover,
\begin{align*}
\det \Sigma(N,P)
&=
\Bigl(N+\frac{N^2}{k}+\frac{mNP}{1+N}\Bigr)\Bigl(cP+\frac{mNP}{1+N}\Bigr)
-\Bigl(\frac{mNP}{1+N}\Bigr)^2\\[4pt]
&=
\Bigl(N+\frac{N^2}{k}\Bigr)\Bigl(cP+\frac{mNP}{1+N}\Bigr)
+\frac{mNP}{1+N}\,cP
>0.
\end{align*}
Thus $\Sigma(N,P)$ has positive trace and positive determinant, hence is positive definite on $(0,\infty)^2$.

If $N=0$ or $P=0$, then one or more rates in \eqref{eq:rate_functions_A} vanish and \eqref{eq:Sigma_explicit} shows
that $\Sigma$ is positive semidefinite but generally singular. For example,
\[
\Sigma(0,P)=\begin{pmatrix}0&0\\0&cP\end{pmatrix},
\qquad
\Sigma(N,0)=\begin{pmatrix}N+\frac{N^2}{k}&0\\0&0\end{pmatrix},
\]
so degeneracy occurs on each axis.
\end{proof}

By continuity of $\Sigma$ and Lemma~\ref{lem:Sigma_posdef_A}, for every compact $K\Subset U$ there exists
$\lambda_K>0$ such that $\xi^\top\Sigma(z)\xi\ge\lambda_K|\xi|^2$ for all $z\in K$ and $\xi\in\mathbb{R}^2$.

\subsection*{Cholesky factorization (two-dimensional Brownian driver)}
\label{subsec:cholesky_factorization_A}

For analysis and simulation it is sometimes preferable to drive the diffusion by a two-dimensional Brownian motion.
On the interior $U=(0,\infty)^2$, this can be achieved by any matrix square root of $\Sigma$, for instance the unique
Cholesky factor with positive diagonal entries.

By Lemma~\ref{lem:Sigma_posdef_A}, for $(N,P)\in U$ there exists a unique lower-triangular matrix
$L_{\rm chol}(N,P)\in\mathbb{R}^{2\times 2}$ with positive diagonal entries such that
\begin{equation}\label{eq:Cholesky_identity_A}
\Sigma(N,P)=L_{\rm chol}(N,P)\,L_{\rm chol}(N,P)^\top .
\end{equation}
An explicit choice is
\begin{equation}\label{eq:Cholesky_explicit}
L_{\rm chol}(N,P)=
\begin{pmatrix}
\sqrt{\Sigma_{11}(N,P)} & 0\\[0.8ex]
\dfrac{\Sigma_{21}(N,P)}{\sqrt{\Sigma_{11}(N,P)}} &
\sqrt{\Sigma_{22}(N,P)-\dfrac{\Sigma_{21}(N,P)^2}{\Sigma_{11}(N,P)}}
\end{pmatrix},
\end{equation}
where $\Sigma_{ij}$ are the entries in \eqref{eq:Sigma_explicit}. In particular, the second diagonal entry is
well-defined and strictly positive on $U$ because $\Sigma$ is positive definite there.

Hence, if $B(t)$ is a standard $\mathbb{R}^2$ Brownian motion, the same diffusion approximation can be written as
\begin{equation}\label{eq:SDE_cholesky_form_A}
dZ(t)=\mu(Z(t))\,dt+\rho\,L_{\rm chol}(Z(t))\,dB(t),
\qquad \rho=\Omega^{-1/2},
\end{equation}
whenever the path remains in $U$.

Equations \eqref{eq:SDE_event_form} and \eqref{eq:SDE_cholesky_form_A} define the same diffusion on $U$ since both
factors satisfy $LL^\top=\Sigma$ on the interior. The event factorization preserves reaction-channel structure and is
natural from the mechanistic standpoint, while the Cholesky factorization reduces the driving noise dimension and is
often convenient for computation. In Section~\ref{sec:numerics} we include finite-sample diagnostics illustrating the
consistency of these two implementations under an absorbed discretization scheme.

\begin{remark}[Mechanistic fingerprint: negative demographic noise correlation]\label{rem:mechanistic_fingerprint_full}
The strict negativity of the cross-covariance
\[
\Sigma_{12}(N,P) = - \frac{mNP}{1+N} < 0 \qquad \text{on } U
\]
is a direct structural consequence of the underlying event setting. 
It serves as a clear fingerprint distinguishing this diffusion from ad hoc diagonal-noise models.
Specifically, the predation-conversion channel $E$ has increment $\Delta_E= (-1,1)^\top$, reflecting the simultaneous loss of one prey and gain of one predator. In the density-dependent diffusion limit, each channel contributes $\lambda_e\Delta_e\Delta_e^\top$ to the instantaneous covariance. For E,
\[
\lambda_E(N,P)\Delta_E\Delta_E^\top = \frac{mNP}{1+N}\begin{pmatrix}
    1 & -1 \\[4pt]
    -1 & 1
\end{pmatrix},
\]
whose off-diagonal entries are strictly negative whenever both species are present. By contrast, diagonal-noise SDEs enforce $\Sigma_{12}\equiv 0$ and therefore cannot capture this event-level coupling. Thus, the sign and magnitude of $\Sigma_{12}$ encode a mechanistic constraint rather than an imposed correlation structure.
\end{remark}

\section{Full-covariance diffusion with absorption: definition and basic properties}
\label{sec:sde_absorption_wellposedness}

We now define the diffusion approximation on the interior $U=(0,\infty)^2$ using the mechanistic drift $\mu$ and
covariance $\Sigma$ from Section~\ref{sec:mu_sigma_factorization}, and we impose absorption at the axes by freezing
the process upon first hitting $\partial U=\{N=0\}\cup\{P=0\}$. This absorbed convention mirrors the
irreversibility of demographic extinction in the underlying CTMC.


Throughout this section we work on a filtered probability space
$(\Omega,\mathcal{F},(\mathcal{F}_t)_{t\ge0},\mathbb{P})$ satisfying the usual conditions.
Let $U:=(0,\infty)^2$ and $\overline U:=[0,\infty)^2$. We fix parameters $k>0$, $m>0$, $c>0$, and a noise scale
$\rho>0$.

Let $\mu:\overline U\to\mathbb{R}^2$ and $\Sigma:\overline U\to\mathbb{R}^{2\times 2}$ be given by
\eqref{eq:mu_explicit}--\eqref{eq:Sigma_explicit}.
On $U$ we use the event-based factor $L_{\rm ev}:\overline U\to\mathbb{R}^{2\times 4}$ defined in
\eqref{eq:Levent_explicit_A}, so that $\Sigma(z)=L_{\rm ev}(z)L_{\rm ev}(z)^\top$ for all $z\in\overline U$
(Lemma~\ref{lem:event_factorization_A}).

\begin{definition}[Interior SDE and absorption time]\label{def:interior_SDE_tau}
Let $z_0=(N_0,P_0)\in U$ and let $W=(W_1,\dots,W_4)^\top$ be a standard $\mathbb{R}^4$ Brownian motion.
An interior strong solution starting from $z_0$ is an $U$-valued, continuous, adapted process
$Z(t)=(N(t),P(t))$ defined up to a (possibly infinite) stopping time $\tau$ such that
\begin{equation}\label{eq:interior_SDE}
Z(t)=z_0+\int_0^{t\wedge\tau}\mu\bigl(Z(s)\bigr)\,ds
+\rho\int_0^{t\wedge\tau} L_{\rm ev}\bigl(Z(s)\bigr)\,dW(s),
\qquad t\ge0,
\end{equation}
and such that
\begin{equation}\label{eq:tau_def_A}
\tau:=\inf\{t>0:\, Z(t)\notin U\}
=\inf\{t>0:\, N(t)=0 \ \text{or}\ P(t)=0\}.
\end{equation}
\end{definition}

\begin{definition}[Absorbed extension]\label{def:absorbed_process}
Given an interior strong solution $(Z,\tau)$ as in Definition~\ref{def:interior_SDE_tau},
the absorbed process $\widehat Z$ is defined for all $t\ge0$ by
\begin{equation}\label{eq:absorbed_extension}
\widehat Z(t):=
\begin{cases}
Z(t), & t<\tau,\\
Z(\tau), & t\ge \tau.
\end{cases}
\end{equation}
We refer to $\tau$ as the absorption time.
\end{definition}

The freezing rule \eqref{eq:absorbed_extension} encodes demographic extinction: once a coordinate reaches $0$, the
corresponding species is considered extinct and the process remains at the boundary thereafter (Figure~\ref{fig:absorption_rule}). We illustrate one set of representative time evolution of the absorbed process in Figure~\ref{fig:absorbed_timeseries}. This convention is
consistent with the CTMC model (Section~\ref{subsec:ctmc_absorbing_axes}), where the axes are absorbing because
birth or predation events are impossible in the absence of individuals.

\begin{figure}[htbp]
    \centering
    \includegraphics[width=0.75\linewidth]{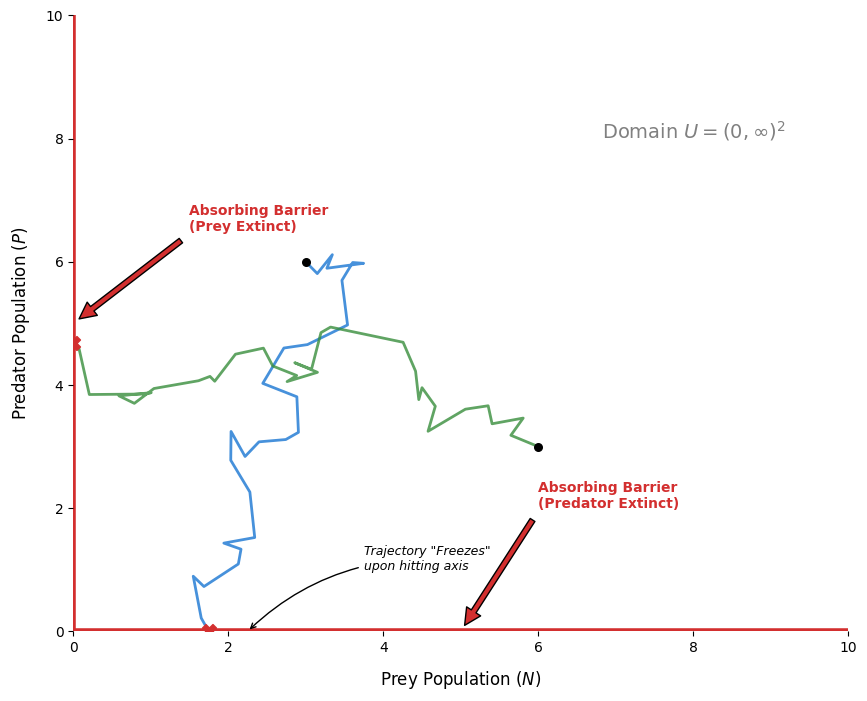}
    \caption{\textbf{Absorption rule.} The model treats extinction as irreversible. Once a trajectory hits either axis ($N=0$ or $P=0$), the diffusion is frozen (marked by 'X'), representing
        the demographic end of a species.}
    \label{fig:absorption_rule}
\end{figure}

\begin{figure}
    \centering
    \includegraphics[width=0.9\linewidth]{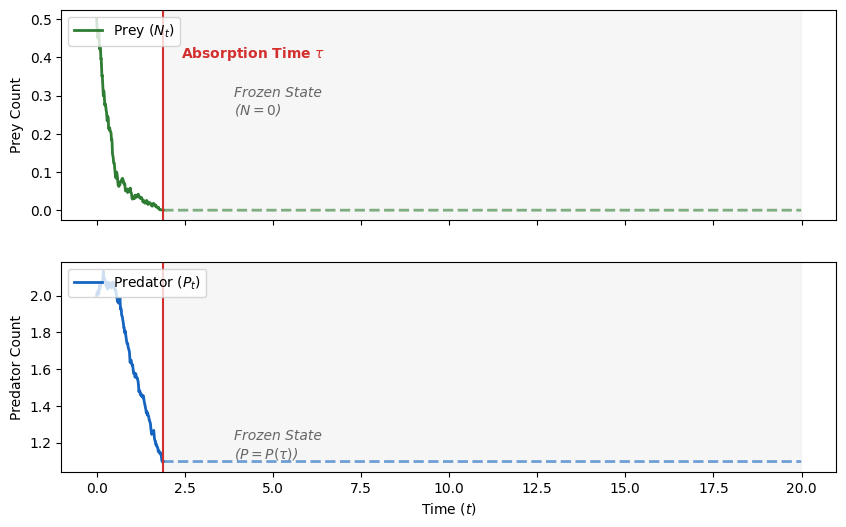}
\caption{\textbf{Time evolution of the absorbed process $\widehat Z(t)$.}
Simulation of a single stochastic trajectory illustrates the freezing convention in Definition~\ref{def:absorbed_process}. The vertical red line marks the absorption time $\tau$ when the prey population ($N_t$, green) hits zero. For all times $t \ge \tau$ (gray shaded region), the entire state vector is frozen at $Z(\tau)$. Note that the predator count ($P_t$, blue) does not decay to zero but remains constant at its value at the moment of prey extinction ($P(\tau)$). This definition formalizes the irreversibility of extinction for the stopped SDE analysis.}
\label{fig:absorbed_timeseries}
\end{figure}

Strong existence and pathwise uniqueness up to absorption. This subsection proves that the interior SDE \eqref{eq:interior_SDE} is strongly well-posed up to the absorption time
$\tau$ and that the absorbed extension \eqref{eq:absorbed_extension} is well-defined. The argument is a standard
localization construction based on truncated globally Lipschitz coefficients; see, e.g.,
\citep{karatzas_brownian_1998,ikeda_stochastic_1981,oksendal_stochastic_2003}.

We now present a fully explicit truncation map and give a detailed localization proof that will be convenient
to cite later.
For the SDE \eqref{eq:SDE_event_form}, the drift vector and the volatility matrix are locally Lipschitz continuous on $U$. Local Lipschitz continuity ensures local well-posedness.
\begin{theorem}[Strong well-posedness up to absorption]\citep{karatzas_brownian_1998}\label{thm:strong_wellposedness_tau}
Fix $z_0\in U$ and let $W$ be a standard $\mathbb{R}^4$ Brownian motion.
Consider the interior SDE
\begin{equation}\label{eq:SDE_tau_recall}
dZ(t)=\mu(Z(t))\,dt+\rho\,L_{\rm ev}(Z(t))\,dW(t),\qquad Z(0)=z_0,
\end{equation}
with absorption time $\tau=\inf\{t>0:\,Z(t)\notin U\}$.
Then there exists an $(\mathcal{F}_t)$-adapted continuous process $Z$ defined on $[0,\tau)$ such that
\eqref{eq:SDE_tau_recall} holds on $[0,\tau)$, and $Z$ is a strong solution in the sense that for each $t<\tau$,
$Z(t)$ is measurable with respect to $\sigma(z_0,W(s):0\le s\le t)$.
Moreover, pathwise uniqueness holds on $[0,\tau)$: if $Z^{(1)}$ and $Z^{(2)}$ are two solutions driven by the
same Brownian motion $W$ with the same initial condition $z_0$, then
\[
\mathbb{P}\bigl(Z^{(1)}(t)=Z^{(2)}(t)\ \text{for all } t<\tau^{(1)}\wedge \tau^{(2)}\bigr)=1.
\]
In particular, the absorbed extension $\widehat Z$ defined by \eqref{eq:absorbed_extension} is well-defined and
unique in law.
\end{theorem}

Non-explosion and moment bounds, we record a quadratic Lyapunov estimate for the interior SDE \eqref{eq:SDE_tau_recall}, which yields non-explosion prior to absorption and uniform polynomial-moment bounds on finite time horizons.
Throughout, write $Z=(N,P)$ and denote by $\Sigma(z):=L_{\rm ev}(z)L_{\rm ev}(z)^{\top}$ the covariance matrix on $U$,
cf.\ Section~\ref{sec:mu_sigma_factorization}. Let $\mathcal{L}$ be the (interior) generator
\begin{equation}\label{eq:generator_A}
(\mathcal{L}f)(z)=\langle \mu(z),\nabla f(z)\rangle+\frac{\rho^2}{2}\,\mathrm{Tr}\bigl(\Sigma(z)\nabla^2 f(z)\bigr),
\qquad z\in U,
\end{equation}
for $f\in C^2(U)$.

\begin{theorem}[Non-explosion and polynomial moment bound]\label{thm:nonexplosion_second_moment}
Let $z_0\in U$ and let $Z$ be the unique strong solution on $[0,\tau)$ from
Theorem~\ref{thm:strong_wellposedness_tau}. Define the radial exit times
\[
\sigma_R:=\inf\{t\ge 0:\ |Z(t)|\ge R\},\qquad R\ge 1,
\]
and the associated explosion time $\zeta:=\lim_{R\to\infty}\sigma_R\in(0,\infty]$.
Then:
\begin{enumerate}[label=(\roman*)]
\item (Non-explosion) One has $\mathbb{P}(\zeta=\infty)=1$. In particular, the maximal lifetime
of the interior SDE is $\tau$, i.e., there is no blow-up prior to hitting $\partial U$.
\item (Polynomial moment bound) 
Assume either $p=2$ or $p\ge 4$. Fix $T>0$. There exists a constant $C_{p,T}<\infty$ depending only on $(k,m,c,\rho, p)$ and $T$ such that for every initial condition $z_0\in U$,
\begin{equation}\label{eq:higher_moment_bound_appD}
\sup_{0\le t\le T}\ \mathbb{E}\bigl[\,|Z(t\wedge\tau)|^p\,\bigr]\ \le\ C_{p,T}\bigl(1+|z_0|^p\bigr).
\end{equation}
\end{enumerate}
\end{theorem}

For completeness, details are provided in Appendix~\ref{app:higher_moments}. Related population models also establish global well-posedness by explicit moment bounds \citep{liang_global_2025}.

\section{Boundary structure and extinction results}
\label{sec:boundary_and_extinction_A}

This section collects qualitative results on extinction for the absorbed diffusion introduced in
Section~\ref{sec:sde_absorption_wellposedness}. 
We emphasize two complementary viewpoints. The first is the absorbed dynamics, in which trajectories are frozen upon reaching the axes and extinction is irreversible. The second is an optional model completion along the boundary, used only for ecological interpretation
of post-extinction dynamics of the remaining species. 

Absorbing axes and reduced boundary dynamics (optional completion). Recall the open quadrant $U:=(0,\infty)^2$ and the absorption time
\begin{equation}\label{eq:tau_def_boundary_A}
\tau:=\inf\{t\ge 0:\ Z(t)\in \partial U\},
\qquad \partial U=\bigl(\{0\}\times[0,\infty)\bigr)\cup\bigl([0,\infty)\times\{0\}\bigr),
\end{equation}
for the interior diffusion $Z(t)=(N(t),P(t))^\top$ constructed in Section~\ref{sec:sde_absorption_wellposedness}.
We write $\widehat Z$ for the absorbed extension (frozen after $\tau$), as in Definition~\ref{def:absorbed_process}.

\begin{proposition}[Absorbing axes for the frozen extension]\label{prop:axes_absorbing_frozen_A}
Let $\widehat Z$ be the absorbed process defined by $\widehat Z(t)=Z(t)$ for $t<\tau$ and
$\widehat Z(t)=Z(\tau)$ for $t\ge \tau$. Then $\partial U$ is absorbing for $\widehat Z$:
if $\widehat Z(t_0)\in\partial U$ for some $t_0\ge 0$, then $\widehat Z(t)\equiv \widehat Z(t_0)$ for all $t\ge t_0$.
In particular, the coordinate axes represent irreversible extinction events in the absorbed model.
\end{proposition}

For ecological interpretation it is sometimes convenient to continue the dynamics after one species has gone
extinct, letting the remaining species evolve according to the corresponding one-dimensional diffusion obtained by
restricting the coefficients to the boundary. This completion is not used in the extinction proofs below
(which are formulated for the absorbed process), and it is introduced only as an auxiliary model component.

\begin{definition}[Axis-restricted diffusions]\label{def:axis_restricted_diffusions_A}
Let $\rho>0$ be the demographic noise amplitude. Define the prey-axis and predator-axis coefficients by restriction:
\[
\mu_N(N,0)=N-\frac{N^2}{k},\qquad
\Sigma_{11}(N,0)=N+\frac{N^2}{k},\qquad N\ge 0,
\]
and
\[
\mu_P(0,P)=-cP,\qquad
\Sigma_{22}(0,P)=cP,\qquad P\ge 0.
\]
The prey-axis diffusion $N^{(0)}$ is the one-dimensional SDE on $[0,\infty)$
\begin{equation}\label{eq:prey_axis_SDE_A}
dN^{(0)}(t)=\Bigl(N^{(0)}(t)-\frac{(N^{(0)}(t))^2}{k}\Bigr)\,dt
+\rho\,\sqrt{N^{(0)}(t)+\frac{(N^{(0)}(t))^2}{k}}\,dB_1(t),
\end{equation}
and the predator-axis diffusion $P^{(0)}$ is the one-dimensional SDE on $[0,\infty)$
\begin{equation}\label{eq:pred_axis_SDE_A}
dP^{(0)}(t)=-cP^{(0)}(t)\,dt+\rho\,\sqrt{cP^{(0)}(t)}\,dB_2(t).
\end{equation}
Here, $B_1,B_2$ are standard one-dimensional Brownian motions (independent, unless specified otherwise).
\end{definition}

\begin{lemma}[Well-posedness on the axes]\label{lem:axis_wellposedness_A}
For any initial condition $N^{(0)}(0)\ge 0$ (resp.\ $P^{(0)}(0)\ge 0$), the SDE \eqref{eq:prey_axis_SDE_A}
(resp.\ \eqref{eq:pred_axis_SDE_A}) admits a pathwise unique strong solution with values in $[0,\infty)$.
Moreover, $0$ is absorbing for both axis-restricted diffusions.
\end{lemma}

\begin{proof}
For \eqref{eq:pred_axis_SDE_A} this is the classical Cox--Ingersoll--Ross (CIR) square-root diffusion with linear drift,
for which strong existence and uniqueness are standard; see, e.g.,
\citep{karatzas_brownian_1998,revuz_continuous_1999}. Since both drift and diffusion vanish at $0$,
the boundary point $0$ is absorbing. The invariance of $[0,\infty)$ follows.

For \eqref{eq:prey_axis_SDE_A}, the drift is locally Lipschitz and the diffusion coefficient $x\mapsto \sqrt{x+x^2/k}$ is H\"older-$1/2$ and of linear growth on $[0,\infty)$. Standard one-dimensional criteria (e.g.\ Yamada--Watanabe-type results for H\"older-$1/2$ diffusion coefficients) yield strong existence/uniqueness and nonnegativity. We refer to \citep{jeanblanc_mathematical_2009,revuz_continuous_1999} for these results. Absorption at $0$ follows
since both drift and diffusion vanish at $0$.
\end{proof}

Positive probability of extinction for all parameters. We next prove that, regardless of parameters, extinction occurs with strictly positive probability starting from any
interior state. Recall the absorption time $\tau=\inf\{t\ge0:\,Z(t)\in\partial U\}$ from \eqref{eq:tau_def_A}.
Throughout, $Z$ denotes the unique strong solution on $[0,\tau)$ from
Theorem~\ref{thm:strong_wellposedness_tau}, and $\widehat Z$ denotes the absorbed extension.

\begin{theorem}[Positive extinction probability]\label{thm:positive_extinction_probability}
Assume $k>0$, $m>0$, $c>0$ and $\rho>0$. Then for every initial condition $z=(N_0,P_0)\in U$,
\begin{equation}\label{eq:positive_ext_prob_A}
\mathbb{P}_z(\tau<\infty)>0.
\end{equation}
Equivalently, $\mathbb{P}_z(\widehat Z(t)\in \partial U \text{ for some }t\ge 0)>0$ for all $z\in U$.
\end{theorem}

For completeness, we include a self-contained proof in Appendix~\ref{app:positive_extinction_probability}.

Almost sure predator extinction in the subcritical regime. We now identify a parameter regime in which predator extinction occurs almost surely. The key assumption is that the
maximal per-capita predator growth rate does not overcome baseline mortality:
\begin{equation}\label{eq:subcritical_condition}
m\le c.
\end{equation}
In this regime, the predator has nonpositive drift throughout the interior, and the square-root demographic noise
remains active down to the boundary, forcing eventual absorption at $P=0$.

\begin{theorem}[Almost sure predator extinction in the subcritical regime]\label{thm:as_predator_extinction_subcritical}
Assume \eqref{eq:subcritical_condition}. Then for every initial condition $z=(N_0,P_0)\in U$, the predator
component of the interior diffusion satisfies
\begin{equation}\label{eq:as_pred_ext_A}
\mathbb{P}_z\Bigl(\inf\{t\ge 0:\ P(t)=0\}<\infty\Bigr)=1.
\end{equation}
Consequently, $\mathbb{P}_z(\tau<\infty)=1$ and the absorbed process $\widehat Z$ satisfies
$\widehat Z(t)\in \partial U$ for all sufficiently large $t$, almost surely.
\end{theorem}

For completeness, we include a self-contained proof in Appendix~\ref{app:cir_hitting}. Figure~\ref{fig:extinction_comparison} compares two stochastic extinction scenarios presented in Theorems~\ref{thm:positive_extinction_probability}--\ref{thm:as_predator_extinction_subcritical}.

\begin{figure}
    \centering
    \includegraphics[width=\linewidth]{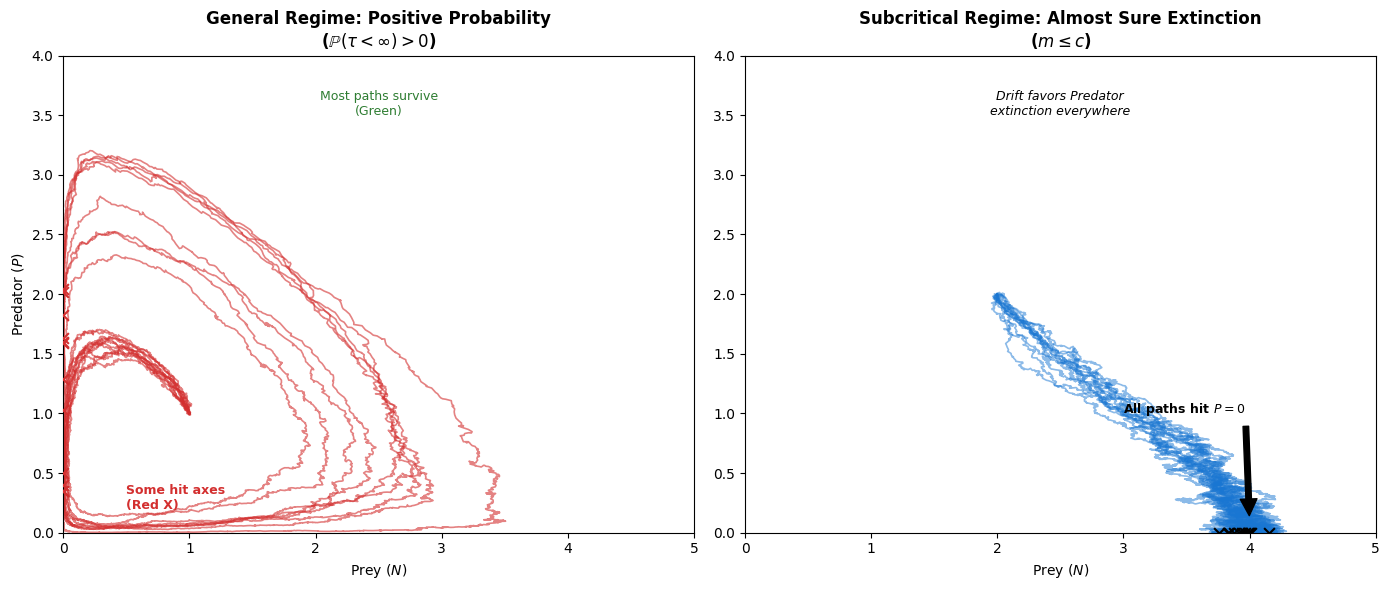}
\caption{\textbf{Stochastic extinction scenarios: positive probability vs.\ almost sure extinction.}
Numerical illustration of the main results in Section~\ref{sec:boundary_and_extinction_A}.
\textbf{(Left)} In the general regime (Theorem~\ref{thm:positive_extinction_probability}), even if the deterministic system predicts coexistence (limit cycles), stochastic fluctuations drive a fraction of paths to extinction (red trajectories marked with 'X'), while others survive transiently (green labels). Extinction is possible but not guaranteed.
\textbf{(Right)} In the subcritical regime $m \le c$ (Theorem~\ref{thm:as_predator_extinction_subcritical}), the predator goes extinct eventually. The drift is uniformly non-positive for the predator, forcing almost sure extinction (all blue paths hit the axis $P=0$) regardless of initial conditions.}
\label{fig:extinction_comparison}
\end{figure}

\section{Numerical methods and diagnostics}
\label{sec:numerics}
This section documents the numerical procedures used to illustrate and corroborate the qualitative statements proved
in Sections~\ref{sec:mech_ctmc}--\ref{sec:boundary_and_extinction_A}.

\subsection{Deterministic check}
\label{subsec:numerics_deterministic_brief}

We include a minimal deterministic sanity check for the Rosenzweig--MacArthur ODE \eqref{eq:RM_ODE}, used in this
paper only for parameter interpretation and as a qualitative reference for bounded coexistence versus deterministic
extinction.
All simulations in this subsection use a standard adaptive Runge--Kutta solver (RK45) with strict
tolerances, and only serve as brief consistency checks. 
We do not pursue detailed deterministic bifurcation computations in this work.
Throughout this numerical check we fix $(m,c,k)=(2.0,0.8,3.0)$ and integrate \eqref{eq:RM_ODE} from three
interior initial conditions.

Since $m>c$ and
\[
N^\ast=\frac{c}{m-c}=\frac{0.8}{2.0-0.8}=\frac{2}{3}\approx 0.6667,
\qquad k=3.0>N^\ast,
\]
the coexistence equilibrium $K_3=(N^\ast,P^\ast)$ exists (Proposition~\ref{prop:equilibria_A}), with
\[
P^\ast=\frac{1+N^\ast}{m}\Bigl(1-\frac{N^\ast}{k}\Bigr)\approx 0.6481,
\qquad\text{so that}\qquad
K_3\approx (0.6667,\,0.6481).
\]
For the same parameters, the Hopf threshold for the trace of the linearization at $K_3$ is
$k_H=1+2N^\ast\approx 2.3333$ (cf.\ Remark~\ref{rmk:hopf_bifurcation_dynamics} in the deterministic analysis).
Because $k=3.0>k_H$, the Jacobian trace at $K_3$ is positive (numerically $\operatorname{tr}J(K_3)\approx 8.89\times10^{-2}$),
so $K_3$ is unstable. Consistent with the local diagnosis, Figure~\ref{fig:hopf} shows that trajectories initialized in $U$ exhibit sustained
oscillations on the simulated time window and approach a closed orbit in the phase plane.

\begin{figure}[hbtp]
    \centering
    \includegraphics[width=0.65\linewidth]{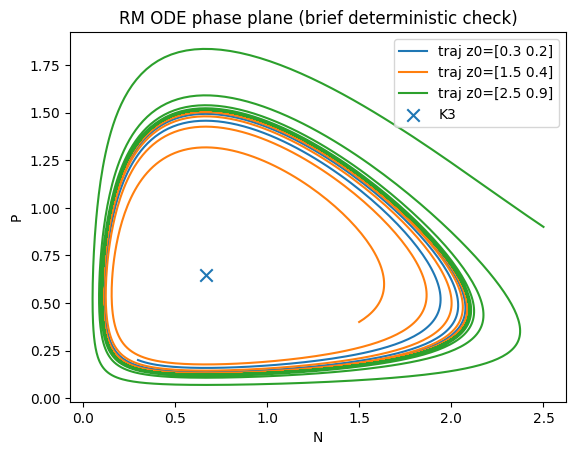}
    \caption{\textbf{Representative trajectories when the coexistence equilibrium exists with the asymptotically stable limit cycle.}}
    \label{fig:hopf}
\end{figure}

As a separate sanity check of the deterministic extinction condition, we simulate an example with $m\le c$:
$(m,c,k)=(0.8,0.8,1.0)$. In this subcritical regime, Theorem~\ref{thm:det_predator_extinction} predicts
$P(t)\to 0$ and $N(t)\to k$ for $N_0>0$, which is confirmed numerically (at $t=200$ we observe $N\approx 1.0$ and
$P\approx 0.0$ up to solver tolerance).

\subsection{Stochastic simulation of the absorbed mechanistic diffusion}
\label{subsec:numerics_sde}

We simulate the mechanistic full-covariance diffusion approximation on $U=(0,\infty)^2$,
\begin{equation}\label{eq:numerics_sde_model_A_rewrite}
dZ(t)=\mu(Z(t))\,dt+\rho\,L(Z(t))\,dB(t),\qquad Z(0)=z\in U,
\end{equation}
with drift $\mu$ and covariance $\Sigma=LL^\top$ given by \eqref{eq:mu_explicit}--\eqref{eq:Sigma_explicit}.
Unless stated otherwise, we fix the noise amplitude at $\rho=\Omega^{-1/2}=0.1$ (corresponding to $\Omega=100$),
and parameters $(k,m,c)=(3.0,2.0,0.8)$ in the super-Hopf coexistence setting.
All simulations use the absorbed convention of Section~\ref{sec:sde_absorption_wellposedness}: once a coordinate hits
$0$, the process is frozen on the boundary. This is the diffusion-level analogue of demographic irreversibility in
the CTMC (Section~\ref{subsec:ctmc_absorbing_axes}).

We employ the EM scheme with time step $\Delta t$:
\begin{equation}\label{eq:EM_update_A_rewrite}
Z_{n+1}=Z_n+\mu(Z_n)\Delta t+\rho\,L(Z_n)\,\Delta B_n,\qquad
\Delta B_n\sim \mathcal N(0,\Delta t\,I),
\end{equation}
where $Z_n\approx Z(n\Delta t)$ and the dimension of the Gaussian increment $\Delta B_n$ matches the chosen factor
$L$ (event factorization: $4$D; Cholesky factorization: $2$D).

Since the target continuous-time model is absorbed at $\partial U$ and then frozen, we enforce absorption at the
discrete level by the following clipping+freezing rule. Define the discrete absorption time
\[
\tau_{\Delta t}:=\inf\{n\Delta t:\ (Z_n)_1\le 0\ \text{or}\ (Z_n)_2\le 0\}.
\]
At the first step where a component becomes nonpositive, we clip componentwise,
\[
(Z_{n+1})_i \leftarrow \max\{(Z_{n+1})_i,\,0\},\qquad i=1,2,
\]
and for all subsequent steps we freeze,
\[
Z_{n+1}=Z_n\quad\text{for all }n\Delta t\ge \tau_{\Delta t}.
\]
This absorption treatment is a numerical implementation of the demographic extinction convention: once prey or
predator density reaches $0$, the corresponding species is extinct and the state remains on the boundary thereafter.

We implement \eqref{eq:numerics_sde_model_A_rewrite} using either the event-based factorization $L=L_{\rm ev}\in\mathbb{R}^{2\times 4}$ \eqref{eq:Levent_explicit_A},
or the Cholesky factorization $L=L_{\rm chol}\in\mathbb{R}^{2\times 2}$ \eqref{eq:Cholesky_explicit}.
Figure~\ref{fig:sde_absorbed_paths} shows representative absorbed EM trajectories under both factorizations. In the Cholesky-based sample path, the prey component crosses $N=0$ boundary, triggering absorption on the finite simulation window. Accordingly, the discrete absorption time
is $\tau_{\Delta t}=87.78$ for the particular realization.
In contrast, the event-based sample path predicts the extinction of the predator population with $\tau_{\Delta t}=51.31$.
The extinction of either population is consistent with the model. The theory
establishes a positive extinction probability in the general parameter regime (Theorem~\ref{thm:positive_extinction_probability}),
so absorption on a finite horizon occurs with positive probability.

To summarize extinction events statistically, we compute the survival curve
$t\mapsto \mathbb{P}(\tau_{\Delta t}>t)$ by Monte Carlo sampling. Figure~\ref{fig:sde_survival_curve} reports an
estimate based on $M=1000$ independent absorbed EM paths (Cholesky factorization, $\Delta t=10^{-2}$, $T=100$). The
resulting decay of the survival probability over the simulation window provides a direct numerical diagnostic of
extinction-by-absorption under demographic noise, complementing the qualitative statement of
Theorem~\ref{thm:positive_extinction_probability}.

\begin{figure}[htbp]
    \centering
    \includegraphics[width=0.45\textwidth]{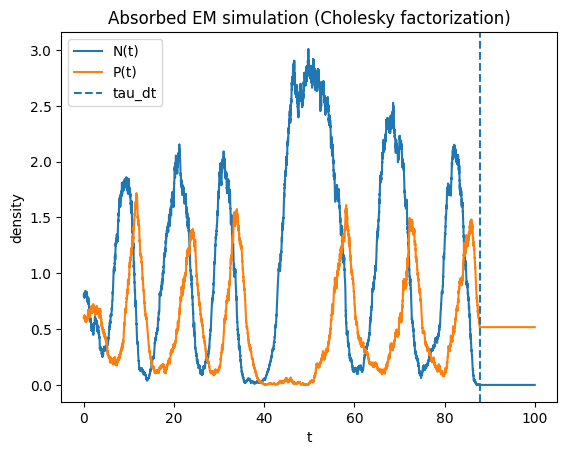}\hspace{1pt}
    \includegraphics[width=0.45\textwidth]{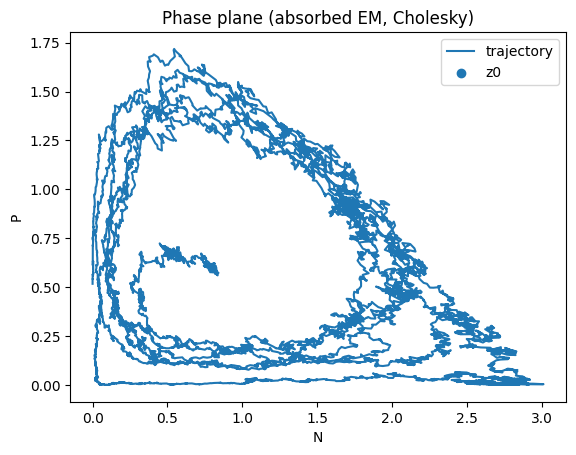}\par
    \includegraphics[width=0.5\textwidth]{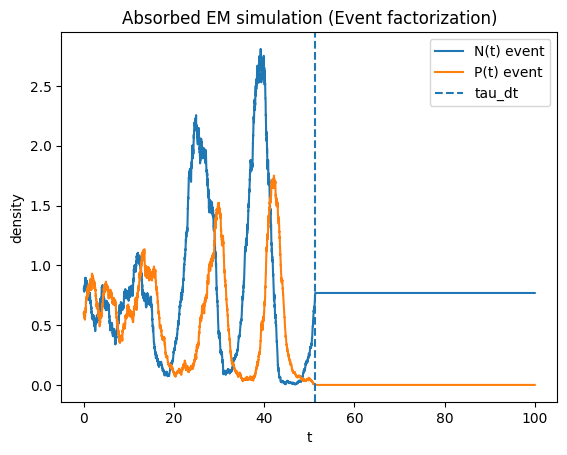}
    \caption{\textbf{Absorbed Euler--Maruyama (EM) simulation (illustrative paths).}
    Representative absorbed EM trajectories for the mechanistic diffusion \eqref{eq:numerics_sde_model_A_rewrite} using
    (top two) the Cholesky factorization $L_{\rm chol}$ and (bottom) the event factorization $L_{\rm ev}$.
    Absorption is enforced by clipping at the first nonpositive coordinate and freezing thereafter.
    Parameters: $(k,m,c)=(3.0,2.0,0.8)$, $\rho=0.1$ ($\Omega=100$), $\Delta t=10^{-2}$, $T=100$, $z_0=(0.8,0.6)$.
    In the displayed single-path realizations, $\tau_{\Delta t}=87.78$ and $\tau_{\Delta t}=51.31$ on $[0,T]$ for the Cholesky-based and event-based sample path.}
    \label{fig:sde_absorbed_paths}
\end{figure}

\begin{figure}[htbp]
    \centering
    \includegraphics[width=0.6\textwidth]{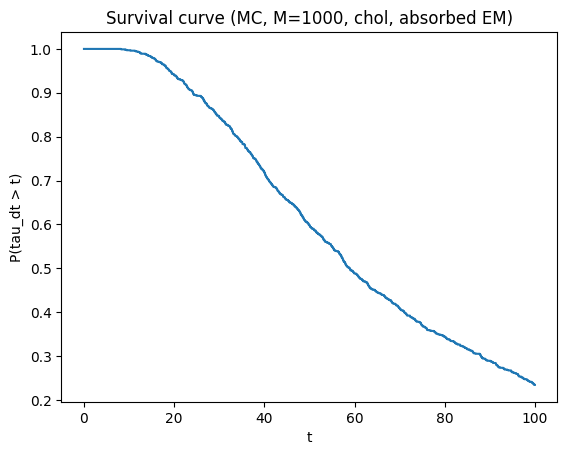}
    \caption{\textbf{Survival curve under absorbed EM (Monte Carlo).}
    Monte Carlo estimate of $t\mapsto \mathbb{P}(\tau_{\Delta t}>t)$ based on $M=1000$ absorbed EM paths
    (Cholesky factorization). Parameters: $(k,m,c)=(3.0,2.0,0.8)$, $\rho=0.1$, $\Delta t=10^{-2}$, $T=50$,
    $z_0=(0.8,0.6)$. The decay over the simulated horizon provides a finite-resolution diagnostic of extinction-by-absorption.}
    \label{fig:sde_survival_curve}
\end{figure}

\subsection{Factorization consistency (event versus Cholesky)}
\label{subsec:numerics_factorization_consistency_A}

On the interior $U$, both the event-based factorization $L_{\rm ev}\in\mathbb{R}^{2\times 4}$ \eqref{eq:Levent_explicit_A} and the Cholesky factorization $L_{\rm chol}\in\mathbb{R}^{2\times 2}$ \eqref{eq:Cholesky_explicit} satisfy
\[
L_{\rm ev}(z)L_{\rm ev}(z)^\top=\Sigma(z)=L_{\rm chol}(z)L_{\rm chol}(z)^\top,\qquad z\in U,
\]
and therefore represent the same diffusion in continuous time. 
Our implementation uses a time discretization
and an absorbed stopping rule at $\partial U$ (Section~\ref{subsec:numerics_sde}). Therefore, it is important to verify that the two factorizations yield consistent finite-resolution statistics at the numerical settings
used in this paper.

We compare absorbed EM simulations \eqref{eq:EM_update_A_rewrite} using
$L=L_{\rm ev}$ (4D Brownian driver) and a$L=L_{\rm chol}$ (2D Brownian driver), under identical model
parameters and discretization:
\[
(k,m,c)=(3.0,2.0,0.8),\qquad \rho=0.1\ (\Omega=100),\qquad
\Delta t=10^{-2},\qquad T=100,\qquad z_0=(0.8,0.6).
\]
For each factorization we generate $M=2000$ independent absorbed EM paths (with independent random seeds across the
two implementations).

We report three diagnostics that are directly tied to models:
\begin{enumerate}[label=(\roman*)]
\item Unconditional mean trajectories. Empirical averages
$t\mapsto \widehat{\mathbb{E}}[N(t)]$ and $t\mapsto \widehat{\mathbb{E}}[P(t)]$.
\item Survival curves. Empirical survival probabilities
$t\mapsto \widehat{\mathbb{P}}(\tau_{\Delta t}>t)$, where $\tau_{\Delta t}$ is the discrete absorption time defined in
Section~\ref{subsec:numerics_sde}.
\item Conditional marginal at $T$. The empirical distribution of $N(T)$ conditioned on survival
$\{\tau_{\Delta t}>T\}$.
\end{enumerate}

Figure~\ref{fig:sde_factorization_consistency} displays the results. The empirical mean trajectories under the two
factorizations are visually indistinguishable at the plotted resolution (left panel), and the survival curves track
each other closely over $[0,T]$ (middle panel). At the terminal time $T=100$, the estimated survivor fractions are
\[
\widehat{\mathbb{P}}(\tau_{\Delta t}>T)=0.244\quad\text{(event factorization)},\qquad
\widehat{\mathbb{P}}(\tau_{\Delta t}>T)=0.245\quad\text{(Cholesky factorization)},
\]
a difference of $0.41\%$ in survival probability at this finite Monte Carlo resolution. Such discrepancies are
expected under independent Monte Carlo sampling across methods and the non-smooth absorbed discretization
(clipping+freezing).
Finally, the conditional histograms of $N(T)$ given survival overlap closely (right panel), indicating no visible
factorization-induced bias in the survival-conditioned distribution at this resolution.

\begin{figure}[htbp]
    \centering
    \includegraphics[width=1.0\textwidth]{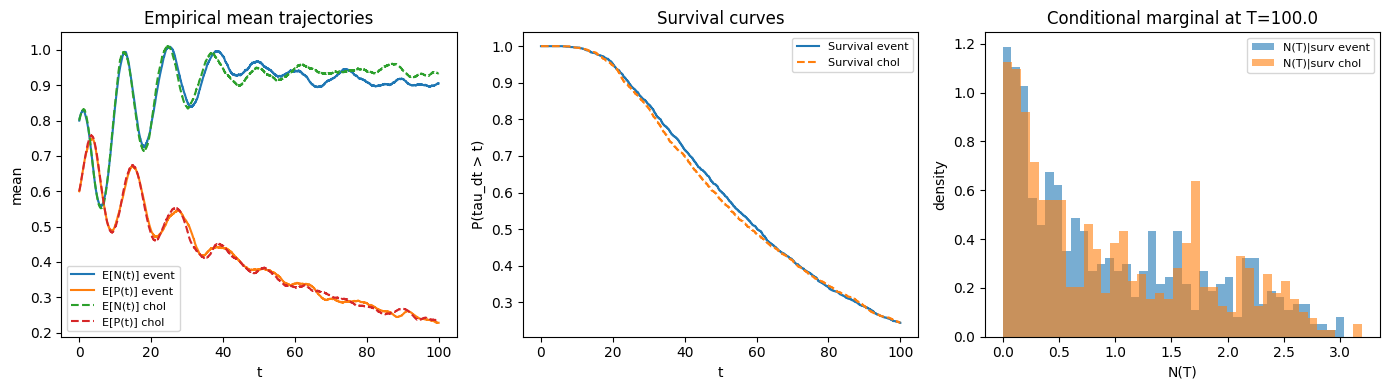}
    \caption{\textbf{Finite-sample consistency of diffusion factorizations.}
    Monte Carlo comparison of absorbed EM simulations using the event-based factorization ($L_{\rm ev}$, 4D noise) and
    the Cholesky factorization ($L_{\rm chol}$, 2D noise), both satisfying $LL^\top=\Sigma$ on $U$.
    \textbf{Left:} empirical mean trajectories $\widehat{\mathbb{E}}[N(t)]$ and $\widehat{\mathbb{E}}[P(t)]$.
    \textbf{Middle:} survival curves $t\mapsto \widehat{\mathbb{P}}(\tau_{\Delta t}>t)$.
    \textbf{Right:} empirical marginal of $N(T)$ conditioned on survival $\{\tau_{\Delta t}>T\}$.
    Parameters: $(k,m,c)=(3.0,2.0,0.8)$, $\rho=0.1$, $\Delta t=10^{-2}$, $T=100$, $z_0=(0.8,0.6)$; $M=2000$ paths per method.
    The estimated survivor fractions at $T=100$ are $0.244$ (event) and $0.245$ (Cholesky).}
    \label{fig:sde_factorization_consistency}
\end{figure}

The event factorization is mechanistically natural (one Brownian driver per reaction channel), whereas the Cholesky
factorization is computationally economical (two-dimensional driver). Figure~\ref{fig:sde_factorization_consistency}
confirms that, under the absorbed EM scheme and at the finite resolution used here, both implementations produce
consistent path statistics and survival-conditioned distributions. We treat the displayed agreement as a
finite-resolution diagnostic rather than as a claim of asymptotic strong or weak convergence rates for absorbed
schemes.

\subsection{Role of covariance: full matrix versus diagonal surrogate}
\label{subsec:numerics_covariance_role_A}

The mechanistic covariance matrix $\Sigma$ in \eqref{eq:Sigma_explicit} has a strictly negative off-diagonal entry
\[
\Sigma_{12}(N,P)=-\frac{mNP}{1+N}<0\qquad \text{for }(N,P)\in U,
\]
which is the instantaneous demographic signature of the coupled predation--conversion event
(Remark~\ref{rem:mechanistic_fingerprint_full}). To probe, at finite numerical resolution, the impact of this
cross-covariance in a representative coexistence setting, we compare the full model to a diagonal surrogate with
matched marginal variances,
\begin{equation}\label{eq:Sigma_diag_rewrite}
\Sigma_{\rm diag}(N,P):=\operatorname{diag}\bigl(\Sigma_{11}(N,P),\Sigma_{22}(N,P)\bigr).
\end{equation}
The diagonal surrogate removes instantaneous correlation ($(\Sigma_{\rm diag})_{12}\equiv 0$) while preserving the pointwise
variances.

We simulate absorbed EM paths (Section~\ref{subsec:numerics_sde}) with time step $\Delta t=10^{-2}$
over horizon $T=100$ using a two-dimensional Brownian driver.
For the full model, we use the Cholesky factor $L_{\rm chol}(N,P)$ satisfying
$L_{\rm chol}L_{\rm chol}^\top=\Sigma$.
For the diagonal surrogate, we use
\[
L_{\rm diag}(N,P):=\operatorname{diag}\bigl(\sqrt{\Sigma_{11}(N,P)},\sqrt{\Sigma_{22}(N,P)}\bigr),
\qquad
L_{\rm diag}L_{\rm diag}^\top=\Sigma_{\rm diag}.
\]
Unless stated otherwise, parameters are fixed at
\[
(k,m,c)=(3.0,2.0,0.8),\qquad \rho=0.1\ (\Omega=100),\qquad z_0=(0.8,0.6),
\]
and we generate $M=2000$ independent absorbed EM trajectories for each model.

We compare two summary diagnostics:
\begin{enumerate}[label=(\roman*)]
\item Survival curves. Monte Carlo estimates of $t\mapsto \widehat{\mathbb{P}}(\tau_{\Delta t}>t)$.
\item Survival-conditioned phase-plane cloud. A snapshot of the particle cloud at $t=100$ restricted to
survivors $\{\tau_{\Delta t}>100\}$, plotted in the $(N,P)$ phase plane.
\end{enumerate}
The goal is not to claim a universal effect of $\Sigma_{12}<0$ across all parameters and noise levels, but rather to provide a transparent ``tested-regime'' diagnostic of the mechanistic correlation structure.

Figure~\ref{fig:covariance_role} shows that, in this tested setting, the survival curves for the full model and the
diagonal surrogate are close over $[0,100]$. At the terminal time $T=100$, the estimated survivor fractions are
\[
\widehat{\mathbb{P}}(\tau_{\Delta t}>100)=0.255\quad\text{(full $\Sigma$)},\qquad
\widehat{\mathbb{P}}(\tau_{\Delta t}>100)=0.247\quad\text{(diagonal surrogate)}.
\]
The results show a $3.14\%$ relative difference in survival probability at this finite Monte Carlo resolution.
The survival-conditioned phase-plane clouds at $t=100$ overlap strongly (right panel). The overlap indicates no visible
factorization-level or correlation-induced qualitative change in the distribution of surviving states at the plotted
resolution.

We report this as an empirical observation in the tested regime. 
In other parameter regimes, for instance, closer to $\partial U$ or at different noise amplitudes $\rho$, the instantaneous correlation structure may have a more pronounced impact on rare-event paths and extinction statistics.

\begin{figure}[htbp]
    \centering
    \includegraphics[width=1.0\textwidth]{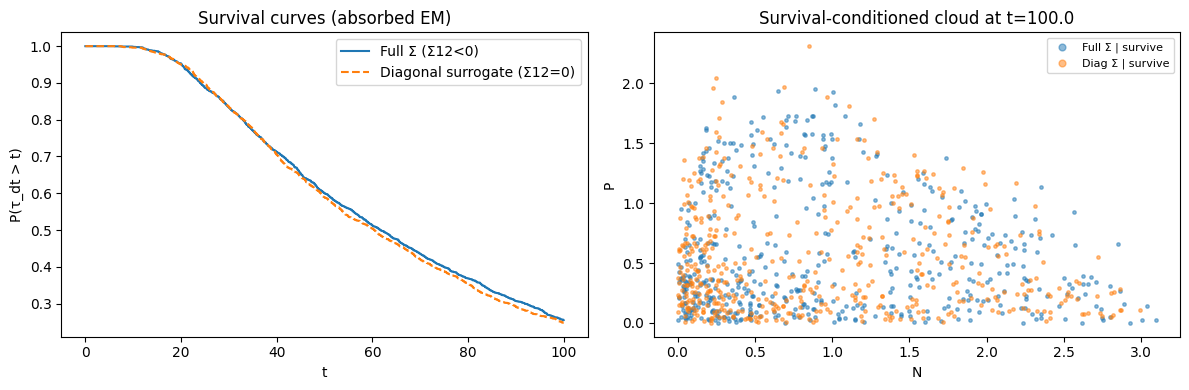}
    \caption{\textbf{Role of demographic covariance: full model vs.\ diagonal surrogate (tested regime).}
    Comparative absorbed EM simulations of the mechanistic diffusion with full covariance $\Sigma$ (blue,
    $\Sigma_{12}<0$) and a diagonal surrogate $\Sigma_{\rm diag}$ (orange, $(\Sigma_{\rm diag})_{12}=0$) with matched
    marginal variances.
    \textbf{Left:} Monte Carlo survival curves $t\mapsto \widehat{\mathbb{P}}(\tau_{\Delta t}>t)$.
    \textbf{Right:} snapshot of the survival-conditioned particle cloud at $t=100$ in the phase plane.
    Parameters: $(k,m,c)=(3.0,2.0,0.8)$, $\rho=0.1$ ($\Omega=100$), $\Delta t=10^{-2}$, $T=100$, $z_0=(0.8,0.6)$;
    $M=2000$ paths per model. The estimated survivor fractions at $T=100$ are $0.255$ (full) and $0.247$ (diagonal).}
    \label{fig:covariance_role}
\end{figure}

The comparison in Figure~\ref{fig:covariance_role} is intended as a finite-resolution diagnostic illustrating how
the mechanistic negative cross-covariance can be probed numerically. It is not a general statement that
correlation is negligible. Rather, it indicates that in the tested regime above, the dominant survival statistics on
the simulated horizon are not strongly sensitive to setting $\Sigma_{12}=0$. A systematic exploration across
$(k,m,c,\rho)$ is beyond the scope of this work.

\section{Discussion}

This work develops a fully mechanistic stochastic Rosenzweig–MacArthur predator–prey model that respects the intrinsic irreversibility of demographic extinction. Rather than postulating an SDE with exogenous diagonal noise, we start from an event-based CTMC on $\mathbb{N}_0^2$ with four elementary channels (prey birth, prey competition death, predator death, and coupled predation–conversion). Absorbing coordinate axes arise naturally at the jump-process level and are preserved in the diffusion approximation by freezing trajectories upon first hitting $\partial U=\{N=0\}\cup\{P=0\}$. 
Under Kurtz density-dependent scaling, the LLN limit recovers the classical Rosenzweig–MacArthur ODE. The CLT/chemical-Langevin limit yields an It\^o diffusion with explicit drift and a full instantaneous covariance matrix.

A central message is that mechanistic event coupling is not a cosmetic modeling choice. The predation–conversion increment $(-1,1)$ necessarily induces a strictly negative cross-covariance
\[
\Sigma_{12}(N,P)=-\frac{mNP}{1+N}<0\quad \text{on }(0,\infty)^2.
\]
The negative cross-covariance is a structural “fingerprint” that cannot be reproduced by diagonal independent-noise surrogates with the same marginal variances. This feature follows directly from the reaction-channel covariance formula $\Sigma(z)=\sum_e \lambda_e(z)\Delta_e\Delta_e^\top$. The feature clarifies when independence assumptions at the SDE level become incompatible with the underlying demographic mechanism.

On the analytical side, we establish strong well-posedness up to absorption, non-explosion, and polynomial moment bounds for the full-covariance diffusion. These results ensure that the absorbed model is mathematically coherent in the ecologically relevant domain and justify interpreting boundary hits as extinction events. Qualitatively, extinction occurs with strictly positive probability from every interior initial condition for all parameter values. In the subcritical regime $m\le c$, predator extinction occurs almost surely. The latter provides a sharp stochastic analogue of deterministic persistence thresholds. It also highlights how demographic noise, combined with boundary degeneracy, enforces eventual absorption even when interior dynamics may appear stable over finite horizons.

From a computational perspective, we record two complementary Brownian factorizations of the diffusion. The first is an event-based factorization $L_{\rm ev}\in\mathbb{R}^{2\times 4}$ that preserves reaction-channel structure, and the second is a two-dimensional Cholesky factorization that is often more economical in simulation. Our absorbed Euler–Maruyama implementation, together with finite-sample diagnostics, confirms practical consistency between these factorizations and offers a transparent workflow for model checking. We also compare the mechanistic diffusion to a diagonal surrogate with matched pointwise variances. The comparison illustrates how the role of $\Sigma_{12}$ can be probed numerically without conflating correlation effects with variance changes. 
In the tested coexistence regime, survival statistics were similar across models. The similarity suggests that covariance effects may be most pronounced in other regimes rather than in typical interior fluctuations.

Several extensions follow naturally. 
First, systematic exploration of how $\Sigma_{12}<0$ alters hitting-time distributions, extinction probabilities would sharpen when mechanistic covariance is essential for risk assessment. Second, developing numerical schemes with improved boundary handling (beyond clipping+freezing) and establishing convergence properties for absorbed diffusion would strengthen the computational foundation. 
Third, the framework readily generalizes to additional mechanistic features, including alternative functional responses, prey refuges, and multi-species networks. Among these features, event coupling may induce richer covariance structures. 
Finally, the model can be connected to data via likelihood-free inference. This approach would enable empirical assessment of whether mechanistic cross-covariance improves fit or predictive performance relative to diagonal-noise approximations.

In conclusion, our results provide a principled pipeline from demographic events to an absorbed full-covariance diffusion. Our model incorporates explicit drift and covariance, rigorous well-posedness up to extinction, and robust extinction statements.
We expose $\Sigma_{12}<0$ as an unavoidable mechanistic consequence of predation–conversion events. Therefore, the paper clarifies the structural limitations of ad hoc diagonal-noise SDEs and supplies a reproducible modeling and simulation blueprint for stochastic extinction analysis in consumer–resource systems.

\begin{appendices}

\section{Non-explosion and moment bounds}\label{app:higher_moments}

This appendix proves Theorem~\ref{thm:nonexplosion_second_moment}.  The argument is a standard Lyapunov--localization
estimate for It\^o diffusion with locally Lipschitz coefficients and at most polynomial growth.  We state all steps
explicitly for completeness.

\begin{proof}
Set the quadratic Lyapunov function $V_q(z):=|z|^q$ for $q=2$ or $q\ge 4$.
For $R\ge 1$ define the bounded stopping time $\theta_R:=\tau\wedge\sigma_R$.

A direct differentiation gives
\begin{equation*}
    \nabla V_q(z) = q \lvert z\rvert^{q-2}x,\quad \nabla^2V_q(z) = q\lvert z\rvert^{q-2}I + q(q-2)\lvert z\rvert^{q-4}zz^\top.
\end{equation*}
Applying It\^o's formula to $V_q(Z(t\wedge \tau))$ gives
\begin{equation*}
    (\mathcal{L}V)(z) = q \lvert z\rvert^{q-2}\langle z,\mu(z)\rangle + \frac{q\rho^2}{2}\lvert z\rvert^{q-2}\operatorname{tr} \Sigma(z) + \frac{q(q-2)\rho^2}{2}\lvert z\rvert^{q-4}z^\top \Sigma(z)z.
\end{equation*}
The positive semidefiniteness of $\Sigma$ yields $z^\top \Sigma(z)z \le \lvert z\rvert^2 \operatorname{tr}\Sigma(z)$. 
Therefore, we have
\begin{equation*}
(\mathcal{L}V_q)(z) \le q\lvert z\rvert^{q-2}\langle z,\mu(z)\rangle + \frac{q(q-1)\rho^2}{2}\lvert z\rvert^{q-2}\operatorname{tr}\Sigma(z).
\end{equation*}
Using the explicit drift $\mu$ and covariance $\Sigma$ from Section~\ref{sec:mu_sigma_factorization}
(and the elementary estimate $NP\le \frac12(N^2+P^2)$), one obtains a constant $C_0<\infty$ such that for all $z\in U$,
\begin{equation}\label{eq:linear_growth_bound}
    \langle z, \mu(z)\rangle \le C_0 (1+\lvert z\rvert^2),\quad \operatorname{tr}\Sigma(z) \le C_0(1+\lvert z\rvert^2).
\end{equation}
Applying \eqref{eq:linear_growth_bound} and the elementary inequality $\lvert z\rvert^{q-2}\le 1+\lvert z\rvert^q$ gives
\begin{equation}\label{eq:LV_bound}
(\mathcal{L}V_q)(z) \le C_q(1+\lvert z\rvert^q)
\end{equation}
for a constant depending on $q$.

By It\^{o}'s formula applied to $V(Z(t\wedge \theta_R))$ and taking expectations,
the local martingale term vanishes and we get, for $t\in[0,T]$,
\[
\mathbb{E}\bigl[V(Z(t\wedge \theta_R))\bigr]
=
V(z_0)+\mathbb{E}\!\left[\int_0^{t\wedge \theta_R} (\mathcal{L}V)(Z(s))\,ds\right]
\le
V(z_0)+C_q\int_0^t \Big( 1 +   \mathbb{E}\bigl[V(Z(s\wedge \theta_R))\bigr] \Big)\,ds,
\]
where we used \eqref{eq:LV_bound}.
Gronwall's inequality yields the following bound that is uniform in $R$:
\begin{equation}\label{eq:V_stopped_bound_A}
\mathbb{E}\bigl[V(Z(t\wedge \theta_R))\bigr]
\le
\big(V(z_0) + C_q t \big)\,e^{C_qt}, \quad t\in [0,T].
\end{equation}

On the event $\{\sigma_R\le T,\ \sigma_R<\tau\}$ one has $|Z(\theta_R)|=|Z(\sigma_R)|\ge R$ and thus
$V(Z(\theta_R))\ge R^2$. Therefore, by \eqref{eq:V_stopped_bound_A},
\[
R^2\,\mathbb{P}(\sigma_R\le T,\ \sigma_R<\tau)
\le
\mathbb{E}\bigl[V(Z(T\wedge \theta_R))\bigr]
\le
(V(z_0)+C_qT)\,e^{C_qT}.
\]
Letting $R\to\infty$ gives $\mathbb{P}(\zeta\le T,\ \zeta<\tau)=0$. Since $T>0$ is arbitrary,
$\mathbb{P}(\zeta<\infty,\ \zeta<\tau)=0$, i.e.\ there is no blow-up prior to boundary hitting.
This proves (i).

For each $m\in \mathbb{N}_+$, consider the stopping time $\tau_m$ of $Z$ that exits $U_m = \bigg(\dfrac{1}{m}, m\bigg) \times \bigg( \dfrac{1}{m}, m\bigg)$.
Reapplying It\^o's formula to $V(Z(t\wedge \tau_m))$ yields 
\begin{equation}\label{eq:V_bound}
    \mathbb{E}\left[V(Z(t\wedge \tau_m))\right] \le \big( V(z_0) + C_qt\big) e^{c_qt}.
\end{equation}
Fix $q\ge 2$. Choose arbitrary $r>q$. \eqref{eq:V_bound} implies that for fixed $t$,
\begin{equation*}
    \sup_{m\in \mathbb{N}_+} \mathbb{E}\left[\lvert Z(t\wedge \tau_m)\rvert ^r \right] \le \infty.
\end{equation*}
Therefore, $\{\lvert Z(t\wedge \tau_m)\rvert^q)\}_{m\in \mathbb{N}_+}$ is uniformly integrable. 
Since $\tau_m \to \tau$ and paths are continuous up to $\tau$, we have $\lvert Z(t\wedge \tau_m)\rvert^q \to \lvert Z(t\wedge \tau)\rvert^q$ almost surely. Vitali's theorem yields
\begin{equation*}
    \mathbb{E}\left[\lvert Z(t\wedge \tau_m)\rvert^q \right] \underset{m\to\infty}{\longrightarrow}
    \mathbb{E}\left[\lvert Z(t\wedge \tau)\rvert^q \right].
\end{equation*}
Taking limit in $m\to \infty$ in \eqref{eq:V_bound} gives
\begin{equation}\label{eq:EV_bound}
    \mathbb{E}\left[V(Z(t\wedge \tau))\right] \le \big( V(z_0) + C_qt\big) e^{c_qt}.
\end{equation}
Taking supermum in $t\in [0,T]$ in \eqref{eq:EV_bound} gives
\begin{equation*}
\sup_{0\le t\le T} \mathbb{E}\left[V(Z(t\wedge \tau))\right] \le 
C_{p,T}\big(1+ V(z_0) \big) = C_{p,T}\big(1+ \lvert z_0\rvert^q \big)
\end{equation*}
for some $C_{p,T}>0$.
\end{proof}

\section{Proof of Theorem~\ref{thm:positive_extinction_probability}}
\label{app:positive_extinction_probability}

In this appendix we give a complete proof of Theorem~\ref{thm:positive_extinction_probability}:
for every $z\in U=(0,\infty)^2$,
\[
\mathbb P_z(\tau<\infty)>0,
\qquad
\tau:=\inf\{t>0:\,N(t)=0\ \text{or}\ P(t)=0\}.
\]

\subsection{Positive density on a compact set away from the boundary}

We work with the interior strong solution $Z(t)=(N(t),P(t))$ of
\begin{equation}\label{eq:B_SDE}
dZ(t)=\mu(Z(t))\,dt+\rho\,L_{\rm ev}(Z(t))\,dW(t),\qquad Z(0)=z\in U,
\end{equation}
where $\mu$ and $\Sigma=L_{\rm ev}L_{\rm ev}^\top$ are given by \eqref{eq:mu_explicit}--\eqref{eq:Sigma_explicit}.
Recall in particular
\begin{equation}\label{eq:B_predator_coeffs}
\mu_2(N,P)=\Bigl(\frac{mN}{1+N}-c\Bigr)P,
\qquad
\Sigma_{22}(N,P)=\Bigl(c+\frac{mN}{1+N}\Bigr)P,
\qquad (N,P)\in U.
\end{equation}

Fix $z=(N_0,P_0)\in U$ and define
\begin{equation}\label{eq:B_choose_compact}
a:=\frac12\min\{N_0,P_0,1\},\qquad
b:=2\max\{N_0,P_0,1\}.
\end{equation}
For $0<a<b$ and $\varepsilon>0$ define
\[
K_{a,b}:=[a,b]\times[a,b]\Subset U,
\qquad
G_{a,b,\varepsilon}:=[a,b]\times(0,\varepsilon)\subset U.
\]
Then $0<a<b<\infty$ and $z\in K_{a,b}$.

\begin{lemma}[Staying inside $K_{a,b}$ for a short time]\label{lem:B_stay_in_K}
There exists $t_0=t_0(z)>0$ such that
\[
\mathbb P_z\bigl(Z(t)\in K_{a,b}\ \text{for all}\ t\in[0,t_0]\bigr)>0.
\]
\end{lemma}

\begin{proof}
Since $z\in\mathrm{int}(K_{a,b})$ and sample paths are continuous, the exit time
$\eta:=\inf\{t\ge0:\,Z(t)\notin K_{a,b}\}$ satisfies $\mathbb P_z(\eta>0)=1$.
Hence $\mathbb P_z(\eta>t_0)>0$ for some $t_0>0$.
\end{proof}

On $K_{a,b}$ the diffusion matrix is uniformly elliptic and smooth.

\begin{lemma}[Uniform ellipticity on $K_{a,b}$]\label{lem:B_uniform_elliptic}
There exists $\lambda_{a,b}>0$ such that for all $(N,P)\in K_{a,b}$ and all $\xi\in\mathbb R^2$,
\[
\xi^\top \Sigma(N,P)\,\xi \ \ge\ \lambda_{a,b}|\xi|^2.
\]
\end{lemma}

\begin{proof}
By Lemma~\ref{lem:Sigma_posdef_A}, $\Sigma$ is positive definite on $U$. By continuity of $\Sigma$ and compactness of
$K_{a,b}\Subset U$, the smallest eigenvalue of $\Sigma$ attains a strictly positive minimum on $K_{a,b}$.
\end{proof}

We use a standard heat-kernel lower bound on compacts. We state it at the level we need.

\begin{lemma}[Small-ball lower bound for the killed transition kernel]\label{lem:B_Aronson_ball}
Fix $t>0$. For each $x\in U$, the killed kernel
\[
P_t(x,dy):=\mathbb P_x(Z(t)\in dy,\ t<\tau)
\]
is absolutely continuous with respect to Lebesgue measure on $U$; we denote its Radon--Nikodym density by
$p_t(x,\cdot)$, i.e.
\[
P_t(x,A)=\int_A p_t(x,y)\,dy,\qquad A\subset U\ \text{Borel}.
\]
Moreover, for each compact $K\Subset U$ there exist constants $c_K(t)>0$ and $r_K(t)>0$ such that
\begin{equation}\label{eq:B_ball_lower_rewrite}
P_t(x,B(y,r))\ \ge\ c_K(t)\,|B(y,r)|,
\qquad \forall\,x,y\in K,\ \forall\,r\in(0,r_K(t)].
\end{equation}
\end{lemma}

\begin{proof}
We show absolute continuity of $P_t(x,\cdot)$.
Fix $x\in U$ and $t>0$. For $R\ge1$ define the bounded domain
\[
D_R:=U\cap B(0,R),
\qquad
\tau_R:=\inf\{s\ge0:\ Z(s)\notin D_R\}.
\]
Consider the killed-in-$D_R$ kernel
\[
P_t^{(R)}(x,A):=\mathbb P_x\bigl(Z(t)\in A,\ t<\tau_R\bigr),\qquad A\subset D_R\ \text{Borel}.
\]
$D_R$ is bounded and the coefficients are smooth and uniformly elliptic on $\overline{D_R}$ away from the axes. Therefore, by standard parabolic theory for uniformly elliptic operators with Dirichlet boundary, $P_t^{(R)}(x,\cdot)$
admits a (Dirichlet) transition density $p_t^{(R)}(x,\cdot)$ with respect to Lebesgue measure on $D_R$:
\[
P_t^{(R)}(x,A)=\int_{A} p_t^{(R)}(x,y)\,dy,\qquad A\subset D_R.
\]
In particular, $P_t^{(R)}(x,\cdot)\ll dy$.

Moreover, for any Borel set $A\subset U$,
\[
P_t^{(R)}(x,A)=\mathbb P_x\bigl(Z(t)\in A,\ t<\tau_R\bigr)
\uparrow
\mathbb P_x\bigl(Z(t)\in A,\ t<\tau\bigr)
=:P_t(x,A),
\qquad R\to\infty,
\]
since $\tau_R\uparrow\tau$ almost surely (continuity of paths and $D_R\uparrow U$).
Therefore, if $A$ is Lebesgue-null, then for every $R$,
$P_t^{(R)}(x,A)=0$, hence $P_t(x,A)=\lim_{R\to\infty}P_t^{(R)}(x,A)=0$.
This shows $P_t(x,\cdot)\ll dy$. By the Radon--Nikodym theorem there exists a density $p_t(x,\cdot)$ such that
\[
P_t(x,A)=\int_A p_t(x,y)\,dy,\qquad A\subset U.
\]

We establish the small-ball lower bound on a compact $K\Subset U$.
Fix a compact set $K\Subset U$. Choose a bounded $C^2$ domain $D$ such that
\[
K\Subset D\Subset U,
\qquad
\mathrm{dist}(D,\partial U)>0.
\]
Let $\tau_D:=\inf\{s\ge0:\ Z(s)\notin D\}$ and define the killed-in-$D$ kernel
\[
P_t^{D}(x,A):=\mathbb P_x\bigl(Z(t)\in A,\ t<\tau_D\bigr),\qquad A\subset D.
\]
Since $D\Subset U$, one has $\{t<\tau_D\}\subset\{t<\tau\}$ and therefore
\begin{equation}\label{eq:Pt_dominates_PtD}
P_t(x,A)\ \ge\ P_t^D(x,A),\qquad x\in U,\ A\subset D.
\end{equation}

On $\overline D$ the diffusion matrix is uniformly elliptic and the coefficients are smooth. Hence, the killed kernel
$P_t^D$ admits a Dirichlet heat-kernel density $p_t^D(x,y)$ on $D\times D$ which is continuous and strictly positive
for each fixed $t>0$. In particular, the continuous function $(x,y)\mapsto p_t^D(x,y)$ attains a strictly positive
minimum on the compact set $K\times K$:
\[
m_K(t):=\inf_{x\in K}\inf_{y\in K} p_t^D(x,y)\ >\ 0.
\]
Let
\[
r_K(t):=\tfrac12\,\mathrm{dist}(K,\partial D)\ >\ 0.
\]
Then for every $y\in K$ and $r\in(0,r_K(t)]$ we have $B(y,r)\subset D$.
Therefore, for $x,y\in K$ and $r\in(0,r_K(t)]$, using \eqref{eq:Pt_dominates_PtD} and integrating the heat kernel,
\[
P_t(x,B(y,r))
\ge P_t^D(x,B(y,r))
= \int_{B(y,r)} p_t^D(x,z)\,dz
\ge \int_{B(y,r)} m_K(t)\,dz
= m_K(t)\,|B(y,r)|.
\]
Thus \eqref{eq:B_ball_lower_rewrite} holds with $c_K(t):=m_K(t)$.
\end{proof}
Choose a point $y_*=(N_*,P_*)$ with $N_*=a$ and $P_*=\varepsilon_*/2$. Take $r>0$ small so that
\[
B(y_*,r)\subset [a/2,2a]\times(0,\varepsilon_*).
\]
Pick a compact $K'\Subset U$ such that
\[
K_{a,b} \cup \overline{B(y_*,r)} \subset K'.
\]
Lemma~\ref{lem:B_Aronson_ball} yields the following bound:
\[
\inf_{x\in K_{a,b}}\mathbb P_x\bigl(Z(t_2)\in B(y_*,r),\ t_2<\tau\bigr)\ge c_{K'}(t_2)\lvert (y_*,r)\rvert>0.
\]
Using the strong Markov property at time $t_1$,
\[
\mathbb P_z\bigl(Z(t_1+t_2)\in B(y_*,r),\ t_1+t_2<\tau\bigr)
\ge \mathbb P_z\bigl(t_1<\tau,\ Z(t_1)\in K_{a,b}\bigr)\cdot c>0.
\]
Since $B(y_*,r)\subset [a/2,2a]\times(0,\varepsilon_*)$ we conclude
\begin{equation}\label{eq:B_enter_G}
\mathbb P_z\bigl(Z(t_1+t_2)\in [a/2,2a]\times(0,\varepsilon_*),\ t_1+t_2<\tau\bigr)>0.
\end{equation}
Define the thin set
\[
G:=G_{a/2,\,2a,\,\varepsilon_*}=[a/2,2a]\times(0,\varepsilon_*).
\]
Equation \eqref{eq:B_enter_G} says: starting from $z$, the process enters $G$ at time $T:=t_1+t_2$ with positive
probability while still alive.

\subsection{One-dimensional comparison inside \texorpdfstring{$G$}{G}}

Fix the thin set $G=[a/2,2a]\times(0,\varepsilon_*)$ and define the exit time
\[
\sigma:=\inf\{t\ge 0:\, Z(t)\notin [a/4,4a]\times(0,2\varepsilon_*)\}.
\]
On $[0,\sigma)$ we have $N(t)\in[a/4,4a]$ and $P(t)\in(0,2\varepsilon_*)$.

On this region we bound the predator drift and diffusion uniformly.

\begin{lemma}[Uniform bounds on the predator coefficients on the strip]\label{lem:B_pred_bounds}
There exist constants $\kappa\in\mathbb R$ and $\alpha>0$, depending only on $(a,m,c)$, such that for all
$(N,P)\in [a/4,4a]\times(0,2\varepsilon_*)$,
\begin{equation}\label{eq:B_drift_bound}
\mu_2(N,P)\le \kappa P,
\end{equation}
and
\begin{equation}\label{eq:B_var_bound}
\Sigma_{22}(N,P)\ge \alpha P.
\end{equation}
\end{lemma}

\begin{proof}
Since $N\mapsto \frac{mN}{1+N}$ is increasing,
\[
\frac{mN}{1+N}-c \ \le\ \frac{m(4a)}{1+4a}-c=: \kappa,
\]
which gives \eqref{eq:B_drift_bound}. Also,
\[
c+\frac{mN}{1+N}\ \ge\ c+\frac{m(a/4)}{1+a/4}=: \alpha>0,
\]
and using \eqref{eq:B_predator_coeffs} gives \eqref{eq:B_var_bound}.
\end{proof}

We compare to a one-dimensional square-root diffusion. Consider the process
\begin{equation}\label{eq:B_CIR_compare}
dY(t)=\kappa Y(t)\,dt+\rho\sqrt{\alpha Y(t)}\,dB(t),\qquad Y(0)=p\in(0,\varepsilon_*),
\end{equation}
where $B$ is a standard one-dimensional Brownian motion.

\begin{lemma}[CIR hits $0$ before leaving $(0,2\varepsilon_*)$ with positive probability]\label{lem:B_CIR_hit}
Let $Y$ solve
\[
dY(t)=\kappa Y(t)\,dt+\rho\sqrt{\alpha Y(t)}\,dB(t),\qquad Y(0)=p\in(0,\varepsilon_*),
\]
with constants $\rho>0$, $\alpha>0$, and $\kappa\in\mathbb R$. Set
\[
\tau_0:=\inf\{t\ge0:\ Y(t)=0\},\qquad
\tau_{2\varepsilon_*}:=\inf\{t\ge0:\ Y(t)\ge 2\varepsilon_*\}.
\]
Then
\[
\mathbb P_p\bigl(\tau_0<\tau_{2\varepsilon_*}\bigr)\ >\ 0.
\]
\end{lemma}

\begin{proof}
By the scaling $X(t):=\frac{4}{\rho^2\alpha}\,Y(t)$, the process $X$ satisfies
\[
dX(t)=\kappa X(t)\,dt+2\sqrt{X(t)}\,dB(t),
\]
i.e.\ $X$ is a square-root diffusion of CIR/BESQ$(0)$ type (zero mean-reversion level). In particular, the boundary
$0$ is accessible and is hit in finite time almost surely: for every $x>0$,
\[
\mathbb P_x\bigl(\inf\{t\ge0:\ X(t)=0\}<\infty\bigr)=1,
\]
see \citep{revuz_continuous_1999}. Returning to $Y$ yields $\mathbb P_p(\tau_0<\infty)=1$ for all $p>0$.

Now fix $\tilde b:=2\varepsilon_*$ and define
\[
u(p):=\mathbb P_p(\tau_0<\tau_{\tilde b}),\qquad p\in(0,\tilde b).
\]
By continuity of paths, $\tau_0\wedge\tau_{\tilde b}>0$ a.s. and $Y(\tau_0\wedge\tau_{\tilde b})\in\{0,\tilde b\}$ a.s. Moreover, by the strong
Markov property, $u$ is the unique bounded solution of the Dirichlet boundary value problem
\[
\mathcal L u=0\ \text{on }(0,\tilde b),\qquad u(0)=1,\ \ u(\tilde b)=0,
\]
where $\mathcal L f(y)=\kappa y f'(y)+\frac{\rho^2\alpha}{2}\,y f''(y)$ is the generator of $Y$ on $(0,\infty)$.
Since $u$ is continuous on $[0,\tilde b]$ with distinct boundary values and satisfies the maximum principle for
one-dimensional diffusions, it follows that $0<u(p)<1$ for all $p\in(0,\tilde b)$.
In particular, for $p\in(0,\varepsilon_*)\subset(0,\tilde b)$ we have $u(p)>0$, i.e.\ $\mathbb P_p(\tau_0<\tau_{2\varepsilon_*})>0$.
\end{proof}

\begin{lemma}[Strip-hitting bound]\label{lem:strip_hit_with_sigma}
Fix $b>0$ and $x=(n,p)\in [a/2,2a]\times(0,b/2)$. Let
\[
\tau_0^P:=\inf\{t\ge0:\ P(t)=0\},\qquad
\tau_b^P:=\inf\{t\ge0:\ P(t)\ge b\},\qquad
\sigma_N:=\inf\{t\ge0:\ N(t)\notin[a/4,4a]\},
\]
and $\theta:=\tau_0^P\wedge\tau_b^P\wedge\sigma_N$.
Assume that on $\{t<\sigma_N\}$,
\[
\mu_2(N(t),P(t))\le \kappa P(t),\qquad \Sigma_{22}(N(t),P(t))\ge \alpha P(t),
\]
for constants $\kappa\in\mathbb R$ and $\alpha>0$. Let $Y$ be the CIR diffusion
\[
dY(t)=\kappa_+Y(t)\,dt+\rho\sqrt{\alpha Y(t)}\,dB(t),\qquad \kappa_+:=\max\{\kappa,0\},
\]
and define $u(r):=\mathbb P_r(\tau_0^Y<\tau_b^Y)$.
Then for every such $x=(n,p)$,
\[
\mathbb P_x\bigl(\tau_0^P<\tau_b^P\wedge\sigma_N\bigr)
\ \ge\
u(p)\;-\;\mathbb P_x\bigl(\sigma_N<\tau_0^P\wedge\tau_b^P\bigr).
\]
\end{lemma}

\begin{proof}
Let $v(r):=\mathbb P_r(\tau_b^Y<\tau_0^Y)=1-u(r)$. Then $v\in C^2((0,b))\cap C([0,b])$ solves
\[
\kappa_+ r\,v'(r)+\frac{\rho^2\alpha}{2}\,r\,v''(r)=0,\qquad v(0)=0,\ v(b)=1,
\]
and satisfies $v'\ge0$.

Set $\theta=\tau_0^P\wedge\tau_b^P\wedge\sigma_N$. Applying It\^o to $v(P(t\wedge\theta))$ and using
$\mu_2\le \kappa P\le \kappa_+P$ together with $\Sigma_{22}\ge \alpha P$ and $v'\ge0$ yields that
$v(P(t\wedge\theta))$ is a bounded supermartingale. Hence
\[
v(p)\ \ge\ \mathbb E_x[v(P(\theta))].
\]
Now decompose into the disjoint events
\[
\text{I}=\{\tau_0^P<\tau_b^P\wedge\sigma_N\},\quad
\text{II}=\{\tau_b^P<\tau_0^P\wedge\sigma_N\},\quad
\text{III}=\{\sigma_N<\tau_0^P\wedge\tau_b^P\}.
\]
On II we have $P(\theta)=b$ hence $v(P(\theta))=1$. On I we have $P(\theta)=0$ hence $v(P(\theta))=0$.
On III we only know $0\le v(P(\theta))\le 1$. Therefore
\[
\mathbb E_x[v(P(\theta))]
=
\mathbb P_x(\text{II})+\mathbb E_x\!\left[v(P(\theta))\,;\text{III}\right]
\ \ge\ \mathbb P_x(\text{II}).
\]
Thus $\mathbb P_x(\text{II})\le v(p)$ and consequently
\[
\mathbb P_x(\text{I})
=
1-\mathbb P_x(\text{II})-\mathbb P_x(\text{III})
\ \ge\
1-v(p)-\mathbb P_x(\text{III})
=
u(p)-\mathbb P_x(\sigma_N<\tau_0^P\wedge\tau_b^P),
\]
which is the claim.
\end{proof}

\begin{lemma}[Uniform strip-hitting probability on $G$]\label{lem:uniform_strip_hit_on_G}
Fix $a>0$ and $\varepsilon_*>0$, and set
\[
G:=[a/2,2a]\times(0,\varepsilon_*),\qquad b:=2\varepsilon_*.
\]
Let $\tau_0^P,\tau_b^P,\sigma_N$ be as in Lemma~\ref{lem:strip_hit_with_sigma}. Assume that there exist constants
$\kappa\in\mathbb R$ and $\alpha>0$ such that on $\{t<\sigma_N\}$,
\[
\mu_2(N(t),P(t))\le \kappa P(t),\qquad \Sigma_{22}(N(t),P(t))\ge \alpha P(t).
\]
Then there exists $q_*>0$ such that
\begin{equation}\label{eq:uniform_strip_hit_qstar}
\inf_{x\in G}\ \mathbb P_x\bigl(\tau_0^P<\tau_b^P\wedge\sigma_N\bigr)\ \ge\ q_*.
\end{equation}
One convenient choice is
\[
q_*:=\frac12\,u(\varepsilon_*),
\qquad
u(r):=\mathbb P_r(\tau_0^Y<\tau_b^Y),
\]
provided $\varepsilon_*>0$ is chosen so that
\begin{equation}\label{eq:sigmaN_small_enough}
\sup_{x\in G}\ \mathbb P_x\bigl(\sigma_N<\tau_0^P\wedge\tau_b^P\bigr)
\ \le\ \frac12\,u(\varepsilon_*).
\end{equation}
\end{lemma}

\begin{proof}
Fix $x=(n,p)\in G$. Apply Lemma~\ref{lem:strip_hit_with_sigma} with $b=2\varepsilon_*$ to obtain
\[
\mathbb P_x\bigl(\tau_0^P<\tau_b^P\wedge\sigma_N\bigr)
\ \ge\
u(p)\;-\;\mathbb P_x\bigl(\sigma_N<\tau_0^P\wedge\tau_b^P\bigr).
\]
Since $p\in(0,\varepsilon_*)$ and $u(\cdot)$ is decreasing in the initial condition (starting closer to $0$ can only
increase the chance of hitting $0$ before $b$), we have $u(p)\ge u(\varepsilon_*)$. Hence
\[
\mathbb P_x\bigl(\tau_0^P<\tau_b^P\wedge\sigma_N\bigr)
\ \ge\
u(\varepsilon_*)\;-\;\mathbb P_x\bigl(\sigma_N<\tau_0^P\wedge\tau_b^P\bigr).
\]
If \eqref{eq:sigmaN_small_enough} holds, then the right-hand side is bounded below by $\frac12u(\varepsilon_*)$,
uniformly in $x\in G$. Taking the infimum over $x\in G$ yields \eqref{eq:uniform_strip_hit_qstar}.
\end{proof}

\begin{proof}[Proof of Theorem~\ref{thm:positive_extinction_probability}]
Let $z\in U$ be arbitrary. By \eqref{eq:B_enter_G}, there exists $T>0$ such that
\[
\mathbb P_z\bigl(Z(T)\in G,\ T<\tau\bigr)\ >\ 0.
\]
Let $b=2\varepsilon_*$ and $\sigma_N$ be the prey-exit time from $[a/4,4a]$ as above, and define the stopping time
\[
\sigma:=\tau_0^P\wedge\tau_b^P\wedge\sigma_N.
\]
On the event $\{\tau_0^P=\sigma\}$ the process hits the boundary $\{P=0\}\subset\partial U$ and hence $\tau<\infty$
under the absorbed convention. Therefore,
\[
\mathbb P_x(\tau<\infty)\ \ge\ \mathbb P_x\bigl(\tau_0^P<\tau_b^P\wedge\sigma_N\bigr),
\qquad x\in G.
\]
By Lemma~\ref{lem:uniform_strip_hit_on_G}, there exists $q_*>0$ such that
\[
\inf_{x\in G}\ \mathbb P_x(\tau<\infty)\ \ge\ q_*.
\]
Apply the strong Markov property at time $T$:
\begin{align*}
\mathbb P_z(\tau<\infty)
&\ge
\mathbb E_z\Bigl[\mathbf 1_{\{Z(T)\in G,\ T<\tau\}}\ \mathbb P_{Z(T)}(\tau<\infty)\Bigr]\\
&\ge
q_*\ \mathbb P_z\bigl(Z(T)\in G,\ T<\tau\bigr)
\ >\ 0.
\end{align*}
This proves $\mathbb P_z(\tau<\infty)>0$ for all $z\in U$.
\end{proof}

\section{A hitting lemma for subcritical extinction}
\label{app:cir_hitting}

This appendix supplies the one-dimensional boundary-hitting input used in
Theorem~\ref{thm:as_predator_extinction_subcritical}. 
We record a domination statement matching the usage in Theorem~\ref{thm:as_predator_extinction_subcritical}.

\begin{lemma}[A convenient domination form]\label{lem:A6_domination_form_A}
Let $P$ be a nonnegative continuous semimartingale on $[0,\tau)$ of the form
\[
dP_t = a_t\,P_t\,dt + dM_t,\qquad P_0>0,
\]
where $a_t\le 0$ for all $t<\tau$ and $M$ is a continuous local martingale satisfying
\[
d\langle M\rangle_t \ \ge\ \rho^2 c\,P_t\,dt
\qquad\text{for all }t<\tau,
\]
for some constants $\rho>0$ and $c>0$.
Then
\[
\tau_0:=\inf\{t\ge 0:\,P_t=0\}<\infty\quad\text{almost surely}.
\]
\end{lemma}

\begin{proof}
We remove the drift by an exponential transform.
Define the adapted finite-variation process
\[
A_t:=\int_0^t a_s\,ds \le 0,
\qquad t<\tau,
\]
and set
\[
X_t:=e^{-A_t}P_t,\qquad t<\tau.
\]
Since $A$ has finite variation, It\^{o}'s formula for products yields, on $[0,\tau)$,
\[
dX_t = e^{-A_t}\,dP_t + P_t\,d(e^{-A_t})
= e^{-A_t}(a_tP_t\,dt+dM_t) - e^{-A_t}a_tP_t\,dt
= e^{-A_t}\,dM_t.
\]
Hence $X$ is a continuous local martingale on $[0,\tau)$.
Moreover, $X_t\ge 0$ for all $t<\tau$ because $P_t\ge 0$ and $e^{-A_t}>0$.
Finally, $X_t=0$ iff $P_t=0$ (since $e^{-A_t}>0$), so the hitting times of $0$ coincide:
\[
\tau_0=\inf\{t\ge0:\,P_t=0\}=\inf\{t\ge0:\,X_t=0\}.
\]
Thus it suffices to show that $X$ hits $0$ in finite time a.s.

We establish a square-root-type lower bound for the quadratic variation of $X$.
Since $X_t=\int_0^t e^{-A_s}\,dM_s + X_0$, we have
\[
d\langle X\rangle_t = e^{-2A_t}\,d\langle M\rangle_t.
\]
Using the assumed bound $d\langle M\rangle_t\ge \rho^2 c\,P_t\,dt$ and $P_t=e^{A_t}X_t$, we obtain
\begin{equation}\label{eq:qv_lower_bound_X}
d\langle X\rangle_t
\ge e^{-2A_t}\,\rho^2 c\,P_t\,dt
= \rho^2 c\,e^{-A_t}X_t\,dt
\ge \rho^2 c\,X_t\,dt,
\qquad t<\tau,
\end{equation}
where we used $e^{-A_t}\ge 1$ because $A_t\le 0$.

Let
\[
Q_t:=\langle X\rangle_{t\wedge\tau_0},\qquad t\ge0.
\]
Then $Q_t$ is continuous and nondecreasing.
By the Dambis--Dubins--Schwarz theorem applied to the continuous local martingale $X_{t\wedge\tau_0}$,
there exists a standard one-dimensional Brownian motion $B$ such that
\begin{equation}\label{eq:DDS_representation}
X_{t\wedge\tau_0} = X_0 + B_{Q_t},\qquad t\ge0.
\end{equation}

Assume for contradiction that $\mathbb P(\tau_0=\infty)>0$ and work on the event $\{\tau_0=\infty\}$.
Then $X_t>0$ for all $t\ge0$ and \eqref{eq:DDS_representation} holds for all $t\ge0$ with $Q_t=\langle X\rangle_t$.

We claim that on $\{\tau_0=\infty\}$ one must have $\langle X\rangle_\infty=\infty$.
Indeed, if $\langle X\rangle_\infty<\infty$, then $Q_t$ converges to a finite limit $Q_\infty$ and hence
$X_t$ converges almost surely to
\[
X_\infty = X_0 + B_{Q_\infty}.
\]
Since $B_{Q_\infty}$ has a continuous distribution conditional on $Q_\infty$, we have
$\mathbb P(X_\infty=0,\ Q_\infty<\infty)=0$.
Thus on $\{\tau_0=\infty,\ Q_\infty<\infty\}$ we have $X_\infty>0$ almost surely.
But then $\int_0^\infty X_s\,ds=\infty$ on that event (because $X_s\to X_\infty>0$), and integrating
\eqref{eq:qv_lower_bound_X} yields
\[
\langle X\rangle_\infty
\ge \rho^2 c\int_0^\infty X_s\,ds
= \infty,
\]
a contradiction. Hence, on $\{\tau_0=\infty\}$ we must have $\langle X\rangle_\infty=\infty$.

Now let
\[
\sigma:=\inf\{u\ge0:\ B_u=-X_0\}.
\]
For Brownian motion, $\sigma<\infty$ almost surely.
Since $Q_t=\langle X\rangle_t$ is continuous, nondecreasing, and diverges to $\infty$ on $\{\tau_0=\infty\}$, there
exists a (random) time $t_\sigma$ such that $Q_{t_\sigma}=\sigma$.
Evaluating \eqref{eq:DDS_representation} at $t_\sigma$ gives
\[
X_{t_\sigma}=X_0+B_{\sigma}=0,
\]
contradicting $\tau_0=\infty$.
Therefore $\mathbb P(\tau_0=\infty)=0$, i.e.\ $\tau_0<\infty$ almost surely.

Since $P_t=e^{A_t}X_t$ with $e^{A_t}>0$, $P$ hits $0$ at the same time as $X$.
Hence $\inf\{t\ge0:\ P_t=0\}<\infty$ almost surely, completing the proof.
\end{proof}

\end{appendices}

\section*{Data Availability:}
Data sharing is not applicable to this article, as no new data was created or analyzed in this study. All simulations are reproducible; code and scripts are provided on demand.

\section*{Author Contribution:}

All authors have accepted responsibility for the entire content of this manuscript and approved its submission.

\section*{Declaration:}
All authors declare no competing interests.

\bibliography{refs}

\end{document}